\documentclass[10pt,reqno]{amsart}

\usepackage{amsmath}
\usepackage{color}
\usepackage{amssymb, amscd, amsmath, amsthm, amsfonts, ae}
\usepackage{graphicx,array}
\usepackage{yhmath} 

\theoremstyle{definition}
\newtheorem{thm}{Theorem}[section]
\newtheorem{lem}[thm]{Lemma}
\newtheorem{defn}[thm]{Definition}
\newtheorem{cor}[thm]{Corollary}
\newtheorem{pro}[thm]{Proposition}
\newtheorem{rem}[thm]{Remark}
\newtheorem{exa}[thm]{Example}
\numberwithin{equation}{section}

\allowdisplaybreaks

\def\C{\mathbb{ C}}
\newcommand{\bbC}{\mathbb{C}}
\newcommand{\Cq}{\mathbb{C}_q}
\newcommand{\CK}{\mathbb{C}\mathcal{K}}

\def\N{\mathbb{ N}}

\def\Z{\mathbb{ Z}}

\def\A{\mathcal{ A}}
\def\B{\mathcal{ B}}

\def\G{\mathcal{ G}}
\def\J{\mathcal{ J}}
\newcommand{\calJ}{\mathcal{J}}
\def\K{\mathcal{ K}}
\def\L{\mathcal{ L}}

\def\P{\mathcal{ P}}
\newcommand{\calR}{\mathcal{R}}

\def\fg{\mathfrak{g}}

\def\eu{\mathfrak{eu}}

\def\m2{\widetilde{M}_2(\C)}

\def\id{\mathrm{id}}

\def\tr{\mathrm{tr}\,}

\newcommand{\AnnV}{{\rm Ann}(V)}

\newcommand{\CplusCop}{\mathbb{C} \oplus \mathbb{C}^{\text{op}}}

\def\sl{\mathfrak{sl}}
\def\gl{\mathfrak{gl}}
\newcommand{\gltnR}{\gl_{2n}\,(R)}
\newcommand{\gtnr}{\mathfrak{g}_{2n,\rho}}
\newcommand{\gtnrR}{\mathfrak{g}_{2n,\rho}(R)}

    \newcommand{\ovlgtnrJ}{\overline{\mathfrak{g}_{2n,\rho}}(J)}
    \newcommand{\undrlgtnrJ}{\underline{\mathfrak{g}_{2n,\rho}}(J)}
    \newcommand{\ovlgtnrJfg}{\overline{\mathfrak{g}_{2n,\rho}}(J(f,g))}
    \newcommand{\undrlgtnrJfg}{\underline{\mathfrak{g}_{2n,\rho}}(J(f,g))}
    \newcommand{\ovlgtnrJfzgz}{\overline{\mathfrak{g}_{2n,\rho}}(J(f_0,g_0))}
    \newcommand{\undrlgtnrJfzgz}{\underline{\mathfrak{g}_{2n,\rho}}(J(f_0,g_0))}
    \newcommand{\ovlgtnrJtftg}{\overline{\mathfrak{g}_{2n,\rho}}(J(\widetilde{f},\widetilde{g}))}
    \newcommand{\undrlgtnrJtftg}{\underline{\mathfrak{g}_{2n,\rho}}(J(\widetilde{f},\widetilde{g}))}

\newcommand{\Gmu}{\mathcal{G}_{\mu}}
\newcommand{\Gnegmu}{\mathcal{G}_{-\mu}}

\newcommand{\Grho}{\mathcal{G}_{\rho}}
\newcommand{\Gzero}{\mathcal{G}_{0}}

\newcommand{\Jone}{J_1}
\newcommand{\Jtwo}{J_2}
\newcommand{\Jthree}{J_3}

\newcommand{\Jab}{J(\alpha,\beta)}
\newcommand{\Jfg}{J(f,g)}
\newcommand{\Jfgone}{J(f,g_1)}
\newcommand{\Jfonegone}{J(f_1,g_1)}
\newcommand{\Jftwogone}{J(f_2,g_1)}
\newcommand{\Jfzgz}{J(f_0,g_0)}

\newcommand{\Jij}{J_{ij}}
\newcommand{\Jik}{J_{ik}}
\newcommand{\Jil}{J_{il}}

\newcommand{\Jji}{J_{ji}}
\newcommand{\Jjk}{J_{jk}}
\newcommand{\Jjl}{J_{jl}}
\newcommand{\Jki}{J_{ki}}
\newcommand{\Jkj}{J_{kj}}
\newcommand{\Jkl}{J_{kl}}

\newcommand{\Jinpi}{J_{i, n+i}}
\newcommand{\Jinpj}{J_{i, n+j}}

\newcommand{\Jknpl}{J_{k, n+l}}

\newcommand{\ovlJinpi}{\overline{J_{i, n+i}}}
\newcommand{\Jinpirho}{\left(J_{i, n+i}\right)_{\rho}}
\newcommand{\Jinpimrho}{\left(J_{i, n+i}\right)_{-\rho}}

\newcommand{\Jnpii}{J_{n+i,i}}
\newcommand{\Jnpij}{J_{n+i,j}}

\newcommand{\Jnpkl}{J_{n+k,l}}

\newcommand{\ovlJnpii}{\overline{J_{n+i, i}}}

\newcommand{\Lrho}{\mathcal{L}_{\rho}}

\newcommand{\MtwoC}{M_2(\mathbb{C})}
\newcommand{\tildeMtwoC}{\widetilde{M_2}(\mathbb{C})}

\newcommand{\rank}{\text{rank}\:}

\newcommand{\varphitilde}{\widetilde{\varphi}}

\newcommand{\amrho}{a_{-\rho}}
\newcommand{\arho}{a_{\rho}}
\newcommand{\ovlamrho}{\,\overline{a_{-\rho}}\,}

\newcommand{\ovla}{\,\overline{a}\,}
\newcommand{\ovlb}{\,\overline{b}\,}

\newcommand{\tildef}{\widetilde{f}}
\newcommand{\tildeg}{\widetilde{g}}

\newcommand{\Jplus}{J_{+}}

\newcommand{\Jrho}{J_{\rho}}
\newcommand{\Jmrho}{J_{-\rho}}

\newcommand{\Rrho}{R_{\rho}}
\newcommand{\JR}{[J,R]}

\newcommand{\RplusRR}{R_+ + [R,R]}
\newcommand{\RplusRRandJ}{\left(R_+ + [R,R]\right) \cap J}

\newcommand{\DeltaC}{\Delta_C}

\newcommand{\aone}{a_1}
\newcommand{\atwo}{a_2}

\newcommand{\bone}{b_1}
\newcommand{\btwo}{b_2}

\newcommand{\cone}{c_1}
\newcommand{\ctwo}{c_2}

\newcommand{\einpi}{e_{i,n+i}}

\newcommand{\fzero}{f_0}
\newcommand{\fone}{f_1}
\newcommand{\ftwo}{f_2}

\newcommand{\foo}{f_{11}}
\newcommand{\fot}{f_{12}}
\newcommand{\fto}{f_{21}}

\newcommand{\fii}{f_{ii}}
\newcommand{\fij}{f_{ij}}
\newcommand{\fik}{f_{ik}}
\newcommand{\fil}{f_{il}}

\newcommand{\fji}{f_{ji}}
\newcommand{\fjj}{f_{jj}}

\newcommand{\fjl}{f_{jl}}
\newcommand{\fjp}{f_{jp}}

\newcommand{\fkj}{f_{kj}}

\newcommand{\fpi}{f_{pi}}

\newcommand{\gzero}{g_0}
\newcommand{\gone}{g_1}
\newcommand{\gtwo}{g_2}

\newcommand{\gii}{g_{ii}}
\newcommand{\gij}{g_{ij}}

\newcommand{\gip}{g_{ip}}
\newcommand{\gji}{g_{ji}}

\newcommand{\gkj}{g_{kj}}

\newcommand{\gkl}{g_{kl}}

\newcommand{\hone}{h_1}
\newcommand{\htwo}{h_2}

\newcommand{\hii}{h_{ii}}
\newcommand{\hij}{h_{ij}}
\newcommand{\hik}{h_{ik}}

\newcommand{\hji}{h_{ji}}

\newcommand{\hkl}{h_{kl}}

\newcommand{\tone}{t_1}
\newcommand{\ttwo}{t_2}

\newcommand{\tonebar}{\overline{t_1}}
\newcommand{\ttwobar}{\overline{t_2}}

\newcommand{\thetaone}{\theta_1}
\newcommand{\thetatwo}{\theta_2}

\title[Finite dimensional irreducible representations of $\fg_{2n,\rho}(\C_q)$]{Finite dimensional irreducible representations of the nullity $2$ centreless core $\fg_{2n,\rho}(\C_q)$}
\author{Sandeep Bhargava, Hongjia Chen and Yun Gao}
\address{S. Bhargava: Mathematics Department, Humber College, Toronto, Canada M9W 5L7}
\thanks{*Corresponding author: H. Chen, hjchen@ustc.edu.cn} 
\address{H. Chen: School of Mathematical Sciences, University of Science and Technology of China,
and Wu Wen Tsun Key Laboratory of Mathematics, Chinese Academy of Science, Hefei 230026, Anhui, P. R. China}
\address{Y. Gao: Department of Mathematics and Statistics, York University, Toronto, Canada M3J 1P3}

\begin{document}

\begin{abstract}
We study the finite-dimensional irreducible representations of the nullity 2 centreless core $\fg_{2n,\rho}(\C_q)$
by investigating the structure of the $\mathrm{BC}_n$-graded Lie algebra $\gtnr(R)$,
where $R$ is a unital involutory associative algebra over a field $k$ of characteristic zero.\vskip5pt

\vskip5pt

\end{abstract}

\maketitle

\section{Introduction}

Extended affine Lie algebras, or EALAs for short, were first introduced
by physicists H{\o}egh-Krohn and Torresani in \cite{HT}
(under the name of quasi-simple Lie algebras), as a
generalization of finite-dimensional simple Lie algebras and affine Kac-Moody Lie algebras
over the complex numbers $\C$.
The structure theory of EALAs has been extensively studied for decades
(See \cite{AABGP}, \cite{ABFP}, \cite{ABP1}, \cite{ABP2}, \cite{N1}, \cite{N2} and \cite{N3}).
In particular, Allison, Azam, Berman, Gao and Pianzola in \cite{AABGP} proved the Kac conjecture
which implies that the root systems of EALAs are examples of extended affine root systems
which were previously introduced by Saito in \cite{S}.

The representations of EALAs are much less well understood.
Different with the finite and affine setting,
there is no successful highest weight theory of EALAs since the lack of a triangular decomposition.
And earlier work considered only the untwisted toroidal Lie algebras and
a few other isolated examples, such as \cite{G1}, \cite{G2}, \cite{W}, \cite{Ya}.
Recently, Billig and Lau in \cite{BL2} constructed irreducible modules
for twisted toroidal Lie algebras and extended affine Lie algebras by combining the
representation theory of untwisted toroidal algebras \cite{B} with the
technique of thin coverings introduced in their paper \cite{BL1}.

As a result of \cite{ABFP} and \cite{N1}, it is clear that,  all but one family of
centreless Lie tori can be constructed as an extension of a twisted multiloop algebra.
Moreover, Lau gave the classification of the finite-dimensional simple modules of multiloop algebras in \cite{L}.
But in general, from an extend affine Lie algebra (or the centreless Lie tori) it is not easy to give the corresponding multiloop algebra structure (see Remark \ref{remarkforTheorem}).
In the present paper, we will study the finite-dimensional irreducible representations
over the centreless core $\fg_{2n,\rho}(\C_q)$ by investigating the structure of the $\mathrm{BC}_n$-graded Lie algebra $\gtnr(R)$.
From any ideal $\J$ of  $\fg_{2n,\rho}(\C_q)$ such that $\fg_{2n,\rho}(\C_q)/\J$ is finite-dimensional and semisimple,
we relate it to an involutive ideal $J$ of $\C_q$.
The corresponding quotient $\C_q/J$ is just copies of an involutive algebra $R$, where $R$ could be $\C$, $\C \oplus \C^{\mathrm{op}}$,
the matrix algebra of order two with different involutions, or the group algebra of the Klein four group, and so on (for more details see Section 3).
In any of these cases, the Lie algebra $\fg_{2n,\rho}(R)$ is a finite-dimensional simple Lie algebra for which the finite-dimensional simple modules are well known.
Hence we obtain a complete result of the finite-dimensional representations of $\fg_{2n,\rho}(\C_q)$.

The paper is organized as follows.
In Section 2, we investigate the structure of the root graded Lie algebra $\gtnr(R)$.
Moreover, we discuss the relation between the involutive ideals
(invariant under involution) of $R$ and ideals of $\gtnr(R)$.
In Section 3, we focus on the elementary quantum torus in two variables with the natural involution.
We define some involutive ideals and study the quotients.
In the Section 4 and Section 5, we talk about our main results about the finite-dimensional semisimple quotients
and finite-dimensional irreducible modules of $\fg_{2n,\rho}(\C_q)$.

Throughout this paper, denote by $\Z$, $\N$, $\C$ and $\C^*$ the sets of integers, positive
integers, complex numbers and nonzero complex numbers respectively.

\section{$\fg_{2n,\rho}(R)$ is a $\mathrm{BC}_n$-graded Lie algebra}

Let us give the definition of $\gtnr(R)$ first.
Let $k$ be a field of characteristic $0$ and $R$ a unital
associative $k$-algebra with involution.
An involution $\,\bar{\cdot}\,$ in this paper means an anti-involution as defined in \cite{KMRT},
that is, a $k-$linear map from $R$ to $R$ satisfying $\overline{ab} =\ovlb \ovla$
and $\bar{\bar{a}} = a$ for all $a, b \in R$.
Given integer $n \geq 1$ and $\rho \in \{ 1, -1 \}$,
let $\gtnr(R)$ be the subalgebra of the Lie algebra $\gl_{2n}(R)$
generated by $\{f_{ij}(a) , 1 \leq i \neq j \leq n\}$, $\{g_{ij}(a) , 1 \leq i \leq j \leq n\}$
and $\{h_{ij}(a) , 1 \leq i \leq j \leq n\}$, where $f_{ij}(a)$, $g_{ij}(a)$ and $h_{ij}(a)$ are defined:
\begin{equation*}
\begin{split}
f_{ij}(a) &=e_{ij}(a)-e_{n+j,n+i}(\bar{a}), \quad 1 \leq i, j \leq n  \\
g_{ij}(a) &=e_{i,n+j}(a)-\rho e_{j,n+i}(\bar{a}), \quad 1 \leq i, j \leq n   \\
h_{ij}(a) & =e_{n+i,j}(a)-\rho e_{n+j,i}(\bar{a}), \quad 1 \leq i, j \leq n.
\end{split}
\end{equation*}

We will frequently need to know how the generators of $\gtnr(R)$
bracket with one another:

\begin{pro}\label{propn:basic brackets}
Given $1 \leq i,j,k,l \leq n$ and $a,b \in R$,
\begin{equation*}
\begin{split}
[f_{ij}(a),f_{kl}(b)] & =\delta_{jk}f_{il}(ab)-\delta_{il}f_{kj}(ba),  \\
[f_{ij}(a),g_{kl}(b)] & =\delta_{jk}g_{il}(ab)+\delta_{jl}g_{ki}(b\bar{a}),  \\
[f_{ij}(a),h_{kl}(b)] & =-\delta_{il}h_{kj}(ba)-\delta_{ik}h_{jl}(\bar{a}b),  \\
[\gij(a),\gkl(b)]     & = 0 = [\hij(a),\hkl(b)],  \\
[g_{ij}(a),h_{kl}(b)] & =\delta_{jk}f_{il}(ab)+\delta_{li}f_{jk}(\bar{a}\bar{b})-\rho
\delta_{jl}f_{ik}(a\bar{b})-\rho \delta_{ki} f_{jl}(\bar{a}b).
\end{split}
\end{equation*}
Also note that
\begin{equation*}
g_{ij}(a)=-\rho g_{ji}(\bar{a}), \quad h_{ij}(a)=-\rho h_{ji}(\bar{a}).
\end{equation*}
\end{pro}
\begin{proof}
Direct calculation yields this proposition.
\end{proof}

\begin{pro}\label{propn:gtnrR is perfect}
If $n \geq 3$, then the Lie algebra $\fg_{2n,\rho}(R)$ is perfect,
i.e.,
$$
[ \fg_{2n,\rho}(R), \ \fg_{2n,\rho}(R) ] = \fg_{2n,\rho}(R).
$$
\end{pro}
\begin{proof}
Since $\gtnrR$ is a subalgebra, $\left[ \gtnrR, \ \gtnrR \right] \subset \gtnrR$.
To show the reverse inclusion, we show that the  generators of $\gtnrR$ can be
expressed as elements of $\left[ \gtnrR, \ \gtnrR \right]$.
Indeed let $a \in R$ and $ 1 \leq i, j \leq n$ with $i \neq j$.
Since $n \geq 3$, we may choose an index $k$ such that $k \neq i$ and $k \neq j$.
But then
\begin{displaymath}
\left[ \fik(a), \ \fkj(1) \right] = \fij(a).
\end{displaymath}
Next take $1 \leq i,j \leq n$, allowing $i = j$.
Again choose $k$ such that $k \neq i$ and $k \neq j$.  Then
\begin{displaymath}
\left[ \fik(a), \ \gkj(1) \right] = \gij(a),
\end{displaymath}
and
\begin{displaymath}
\left[ \fkj(-1), \ \hik(a) \right] = \hij(a).  \qedhere
\end{displaymath}
\end{proof}

We begin by examining the structure of $\fg_{2n,\rho}(R)$. Let
$R_{\pm} = \{a \in R \,|\,\bar{a} = \pm a \}$. Then
\begin{equation}\label{eqn:[R,R] in R-}
R=R_+ \oplus R_- \text{ and }  R_+ + [R,R]=R_+ \oplus (R_- \cap [R,R]).
\end{equation}

We will make use of the following lemma:
\begin{lem}\label{lemma:Wonenburger}
Let $\A$ be an associative $k$-algebra and $\varphi$ an involution
on $\A$,  
then the subspace
\begin{equation*}
L = \{ a \in \A \;|\; \varphi(a) =-a\}
\end{equation*}
of $\A$ is a Lie algebra under the operation $[a,b]=ab-ba$.
\end{lem}

Let
\begin{equation*}
M=\left(
\begin{array}{cc}
0 & I_n \\
\rho I_n & 0 \\
\end{array}
\right) \in M_{2n}(R).
\end{equation*}
Then $M$ is an invertible {\small $2n$} by {\small $2n$} matrix and
$\overline{M}^t=\rho M$, where $A^t$ means the transpose of a matrix $A$.
Using the matrix $M$, define a map
$$
^*:M_{2n}(R) \rightarrow M_{2n}(R)\text{ by } A^*=M^{-1}\overline{A}^tM.
$$
Since $\overline{M}^t=\rho M$,
\begin{equation*}
\begin{split}
(A^*)^*= & (M^{-1}\overline{A}^tM)^*= M^{-1}( \overline{M^{-1}
\overline{A}^tM})^t M =M^{-1} \overline{M}^t A (\overline{M}^t)^{-1} M \\
= & M^{-1}\rho M A \rho^{-1} M^{-1} M =A,
\end{split}
\end{equation*}
and
\begin{equation*}
(AB)^*=M^{-1}\overline{AB}^tM = M^{-1}\overline{B}^t \overline{A}^t
M=B^*A^*.
\end{equation*}
That is, $^*$ is an involution on the associative algebra
$M_{2n}(R)$. By Lemma \ref{lemma:Wonenburger},
\begin{equation*}
\L_{\rho}=\{A \in M_{2n}(R)\,|\,A^*=-A \}.
\end{equation*}
is a Lie subalgebra of $\gl_{2n}(R)$. The general form of a
matrix in $\L_{\rho}$ is
\begin{equation}\label{description:matrix in Lrho}
\left(  \begin{array}{cc}
P & S \\
T & -\bar{P}^t\\
\end{array} \right) \quad \text{with   } \overline{S}^t=-\rho S \quad
\text{and}  \quad     \overline{T}^t=-\rho T,
\end{equation}
where $P,S,T \in M_n(R)$.

For $i = 2, \ldots, n$ and $a \in R$, observe that $\fii(a) = \left(
\fii(a) - \foo(a) \right) + \foo(a) $.  So, using the description in
(\ref{description:matrix in Lrho}),  a spanning set for $\Lrho$ is
\begin{equation*}
\begin{split}
\big\{ \foo(a) | a \in R \big\}
\cup &
\Big( \bigcup_{i=2}^n \left\{ \fii(a) - \foo(a) | a \in R \right\} \Big) \cup
\Big( \bigcup_{1 \leq i \neq j \leq n} \left\{ \fij(a) | a \in R \right\} \Big) \\
\cup &
\Big( \bigcup_{1 \leq i < j \leq n} \left\{ \gij(a) | a \in R \right\} \Big)
\cup \Big( \bigcup_{1 \leq i \leq n} \left\{ \gii(a) | a \in R \right\} \Big)  \\
\cup &
\Big( \bigcup_{1 \leq i < j \leq n} \left\{ \hij(a) | a \in R \right\} \Big) \cup
\Big( \bigcup_{1 \leq i \leq n} \left\{ \hii(a) | a \in R \right\} \Big).
\end{split}
\end{equation*}

Let $\G_{\rho}$ denote the commutator algebra of $\L_{\rho}$, that
is, $\G_{\rho}=[\L_{\rho}, \L_{\rho}]$. For given $Y=(y_{ij}) \in M_{2n}(R)$,
the trace of $Y$ is defined to be $\tr Y = \sum\limits_{i=1}^{2n} y_{ii}
\in R$. Then we have the following Lemma:

\begin{lem}\label{lemma:Grho and trace}
For $n \geq 2$, the elements of $\Grho$ are precisely the elements of $\Lrho$ whose
trace lies in $[R,R]$, i.e.,
$$
\G_{\rho}=\bigl\{Y \in \L_{\rho}\,|\,\tr(Y) \equiv 0 \bmod
[R,R]\bigr\}.
$$
\end{lem}
\begin{proof}
For any $A, B \in \gltnR$, $\tr([A,B]) \in [R,R]$ implies that
$$
\Grho \subset \left\{ Y \in \Lrho | \tr(Y) \equiv 0 \bmod [R,R] \right\}.
$$

To show the reverse inclusion, we begin by observing that a typical
element $Y$ in $\Lrho$ can be described as
\begin{equation*}
\begin{split}
Y  = & \foo(p) + \sum_{i=2}^n (\fii(q_i) - \foo(q_i)) + \sum_{1 \leq i \neq j \leq n}
       \fij(r_{ij}) + \sum_{1 \leq i < j \leq n} \gij(s_{ij}) \\
& + \sum_{1 \leq i \leq n} \gii(t_i) + \sum_{1 \leq i < j \leq n} \hij(u_{ij})
  + \sum_{1 \leq i \leq n} \hii(v_i),
\end{split}
\end{equation*}
where $p$, $q_i$, $r_{ij}$, $s_{ij}$, $u_{ij}$, $t_i$ and $v_i$
are elements of $R$.  Also note that
\begin{equation*}
\tr(Y) = \tr( \foo(p)) + \sum_{i=2}^n \Big( \tr( \fii(q_i) ) - \tr( \foo(q_i) ) \Big)
= p - \overline{p} = \tr( \foo(p) ).
\end{equation*}
So let us take such an element $Y$ in $\Lrho$ such that its trace lies in $[R,R]$,
that is, $p - \overline{p} \in [R,R]$. We will show that each term in the above sum
lies in $\Grho = [ \Lrho, \, \Lrho ]$.  Indeed, observe that
\begin{enumerate}
\item
for $ 2 \leq i \leq n$, $\fii(q_i) - \foo(q_i) = \left[ f_{i1} (q_i), \ f_{1i} (1) \right]$,
\item
for $1 \leq i \neq j \leq n$, $\fij(r_{ij}) = \left[ \fii(r_{ij}),  \fij(1) \right]$,
\item
for $1 \leq i < j \leq n$, $\gij(s_{ij}) = \left[ \fii(s_{ij}), \gij(1) \right]$,
$\hij(u_{ij}) = \left[ \fjj(-1), \hij(u_{ij}) \right]$, 
\item
for $1 \leq i \leq n$, $\gii(t_i) = \left[ \fii(\frac{1}{2}),  \gii(t_i) \right]$ and
$\hii(v_i) = \left[ \fii(-\frac{1}{2}), \hii(v_i) \right]$.
\end{enumerate}

A slightly fancier argument is needed to show that $\foo(p) \in [\Lrho, \, \Lrho]$.
Since $p \in R$, there exist unique $x \in R_+$ and $y \in R_-$ such that $p = x+y$.
So $p - \overline{p} = 2y \in R_- \cap [R,R]$.
Since $R_+ \oplus \left( R_- \cap [R,R] \right) = R_+ + [R,R]$, we can rewrite $p$ as
$p = a + \sum_k [b_k, \, c_k]$ for some $a \in R_+$ and $b_k, c_k \in R$.
Since $a \in R_+$, we have now
\begin{equation*}
\begin{split}
\foo(p) = & \foo(a) + \sum_k \foo\left( [b_k, \, c_k] \right) \\
= & [ f_{12}(a), \ f_{21}(1/2) ] + [ g_{12}(1/2), \ h_{21}(a) ] + \sum_k \left[ \foo(b_k),  \foo(c_k) \right]
\end{split}
\end{equation*}
is also a member of $[\Lrho, \, \Lrho] = \Grho$ like the other summands of $Y$.
Hence $Y \in [\Lrho, \, \Lrho]$ and
\begin{equation*}
\left\{ Y \in \Lrho | \tr(Y) \equiv 0 \bmod [R,R] \right\} \subset \Grho.  \qedhere
\end{equation*}
\end{proof}

$\Grho$ has an abelian subalgebra
$\mathcal{H}= \Bigl\{\sum\limits_{i=1}^na_i(e_{ii}-e_{n+i,n+i})|a_i \in k\Bigr\}$
of dimension $n$.
The linear functions $\epsilon_i$ in the dual space $\mathcal{H}^*$,
defined by
$\epsilon_i\biggl(\sum\limits_{j=1}^na_j(e_{jj}-e_{n+j,n+j}) \biggr)=a_i$
for $i=1, \ldots, n$, permit us to recognize  $\Grho$'s structure as
that of a graded Lie algebra:
\begin{equation}\label{eqn:Grho is root-graded}
\G_{\rho}=\G_0 \oplus \sum_{i \neq j} \G_{\epsilon_i-\epsilon_j} \oplus
\sum_{i<j}(\G_{\epsilon_i+\epsilon_j} \oplus \G_{-\epsilon_i-\epsilon_j})
\oplus \sum_{i} (\G_{2\epsilon_i} \oplus \G_{-2\epsilon_i})
\end{equation}
where
$\G_{\alpha}=\{x\in \G_{\rho}|[h,x]=\alpha(h)x, \text{ for all } h \in \mathcal{H}\}$
for a given $\alpha \in \mathcal{H}^*$.  Moreover,
\begin{equation}\label{description:nonzero root spaces of Grho}
\G_{\epsilon_i-\epsilon_j}= f_{ij}(R) \; (i \neq j), \quad  \G_{\epsilon_i+\epsilon_j}=g_{ij}(R), \quad
\G_{-\epsilon_i-\epsilon_j}=h_{ij}(R),
\end{equation}
and
\begin{equation}\label{description:zero root space of Grho}
\G_0=f_{11}(R_+ \oplus (R_- \cap [R,R])) \bigoplus \sum_{i}(f_{ii}-f_{11})(R),
\end{equation}
where the right hand sides of (\ref{description:nonzero root spaces
of Grho}) and (\ref{description:zero root space of Grho}) are
defined in a natural way.

Now we need
\begin{lem}\label{lemma:G0 generated by nonzero Galpha}
For $n \geq 2$,
\begin{equation*}
\G_0 =\sum_{\mu \in \Delta_{C}} [\G_{\mu},\G_{-\mu}],
\end{equation*}
where $\Delta_C$ denotes the root system of type $C$.
\end{lem}
\begin{proof}
By the definition of $\G_\alpha$, we have that $\sum\limits_{\mu \in \DeltaC} [\Gmu, \Gnegmu] \subset \Gzero$.
The reverse inclusion follows from the proof of
Lemma \ref{lemma:Grho and trace} and
\begin{equation*}
\foo([c,d])  = [ \fto(c),  \fot(d) ] - [ \fto(cd),  \fot(1) ]
\in \left[ \G_{\epsilon_2 - \epsilon_1},  \G_{-\left(\epsilon_2 - \epsilon_1\right)}
\right], \;  \forall c,d \in R.  \qedhere
\end{equation*}
\end{proof}

\begin{pro}
If $n \geq 2$, then $\G_\rho$ is a $\mathrm{BC}_n$-graded Lie algebra with
grading subalgebra of type $C_n$ if $\rho=-1$ and type $D_n$ if
$\rho=1$. Furthermore, $\G_\rho=\fg_{2n,\rho}(R)$.
\end{pro}
\begin{proof}
The first part of the theorem comes from (\ref{eqn:Grho is
root-graded})-(\ref{description:zero root space of Grho}), Lemma
\ref{lemma:G0 generated by nonzero Galpha} and Example 1.16 in \cite{ABG}.
For more information on $\mathrm{BC}_n$-graded Lie algebras, see \cite{ABG}.
Now, by Lemma \ref{lemma:G0 generated by nonzero Galpha} again,
we get that $\G_\rho$ and $\fg_{2n,\rho}(R)$ have the same generators.
\end{proof}

\begin{rem}
We gave two definitions of the Lie algebra $\fg_{2n,\rho}$:
It is easy to show the perfectness from the definition given by generators.
The second one can be used for writing the elements of $\fg_{2n,\rho}$ explicitly and for presenting the $\mathrm{BC}$-graded structure of $\fg_{2n,\rho}$.
\end{rem}

Thus,
\begin{cor}\label{cor:gtnrR} If $n \geq 2$, then
\begin{equation*}
\fg_{2n,\rho}(R) =f_{11}(R_+ + [R,R]) \oplus
\sum_{i=2}^n(f_{ii}-f_{11})(R) \oplus \sum_{i \neq j} f_{ij}(R)
\oplus \sum_{i \leq j} g_{ij}(R) \oplus \sum_{i \leq j} h_{ij}(R).
\end{equation*}
\end{cor}

We get some familiar Lie algebras in $\fg_{2n,\rho}(R)$ for
different choices of $R$.

\begin{exa}\label{eg:gtnrC}
\begin{enumerate}
\item
Let $R = \C$ with the identity map as its involution.
Then $\fg_{2n,\rho}(\C)$ is a semisimple Lie algebra for $n \geq \frac{3+\rho}{2}$ in \cite{H}.
More explicitly, $\fg_{2n,-1}(\C)$ is the simple Lie
algebra of type $C_n$ for $n \geq 1$ and $\fg_{2n,1}(\C)$ is the
simple Lie algebra of type $D_n$ for $n \geq 3$. For $n=2$,
$\fg_{4,1}(\C)$ is the semisimple Lie algebra of type $D_2$, which
is two copies of $\sl_2(\C)$.

\item
Let $R = \C[t^{\pm 1}]$ be the Laurent polynomial and the involution is also the identity map.
Then $\fg_{2n,\rho}(\C[t^{\pm 1}]) \cong \fg_{2n,\rho}(\C) \otimes \C[t^{\pm 1}]$ is
the corresponding centerless core of an untwisted affine Kac-Moody algebra in \cite{K}.
\end{enumerate}
\end{exa}

\begin{exa}\label{eg:gtnrM}
Let $R = M_m(k)$ be the $m$ by $m$ matrix algebra with
involution $\bar{\;\;}$ defined by $\overline{A} = A^t$. Then $\fg_{2n,\rho}(M_m(k))$ is
isomorphic to $\fg_{2nm,\rho}(k)$ for $n \geq 2$.
\end{exa}
\begin{proof}
The isomorphism $\varphi$ is given by
$$
f_{ij}(E_{kl}) \mapsto  f_{(i-1)m+k,(j-1)m+l}(1), \text{ for all } 1
\leq i \neq j \leq n, \; 1 \leq k , l \leq m\,;
$$
and for all $1 \leq i < j \leq n, \; 1 \leq k , l \leq m\,$
$$
g_{ij}(E_{kl}) \mapsto  g_{(i-1)m+k,(j-1)m+l}(1),\;\; h_{ij}(E_{kl})
\mapsto h_{(i-1)m+k,(j-1)m+l}(1);
$$
if $\rho=-1$, for all $1 \leq i  \leq n, \; 1 \leq k \leq l \leq m$,
$$
g_{ii}(E_{kl}) \mapsto  g_{(i-1)m+k,(i-1)m+l}(1),\;\; h_{ii}(E_{kl})
\mapsto h_{(i-1)m+k,(i-1)m+l}(1) ;
$$
if $\rho=1$, for all $1 \leq i  \leq n, \; 1 \leq k < l \leq m$,
$$
g_{ii}(E_{kl}) \mapsto  g_{(i-1)m+k,(i-1)m+l}(1),\;\; h_{ii}(E_{kl})
\mapsto h_{(i-1)m+k,(i-1)m+l}(1)
$$
where $E_{ij}$ is the $m \times m$ matrix with $1$ at $(i, j)$
position and $0$ otherwise.
\end{proof}

\begin{rem}
For $n \geq 2$ and $R=M_m(k)$, from Corollary \ref{cor:gtnrR} we
have
$$
\fg_{2n,\rho}(M_m(k)) =\sum_{i, j} f_{ij}(M_m(k)) \oplus \sum_{i
\leq j} g_{ij}(M_m(k)) \oplus \sum_{i \leq j} h_{ij}(M_m(k))
$$
and
$$
\dim \fg_{2n,\rho}(M_m(k))=2(mn)^2-\rho m n=\dim \fg_{2nm,\rho}(k).
$$
\end{rem}

\begin{exa}\label{eg:gtnrSplusSop}
Let $S$ be a unital associative $\C$-algebra and $R = S \oplus
S^{op}$ with involution $\overline{(a, b)} = (b, a)$ for $a,b \in
S$. Then $\fg_{2n,\rho}(R)$ is isomorphic to $\sl_{2n}(S)$ for all $n
\geq 2$. In particular, if $S = M_m(\C)$, where $m \geq 1$, then
$\fg_{2n,\rho}(M_m(\C) \oplus M_m(\C)^{op}) \cong \sl_{2nm}(\C)$.
\end{exa}
\begin{proof}
Let $\varphi$ : $\fg_{2n,\rho}(R) \rightarrow \sl_{2n}(S)$ be defined
by
$$
f_{ij}(a,0) \mapsto
\frac{1}{2}\Bigl(e_{2i-1,2j-1}(a)+e_{2i,2j}(a)+\sqrt{-1}(e_{2i-1,2j}(a)
-e_{2i,2j-1}(a))\Bigr),
$$
$$
f_{ji}(0,a) \mapsto
\frac{1}{2}\Bigl(-e_{2i-1,2j-1}(a)-e_{2i,2j}(a)+\sqrt{-1}(e_{2i-1,2j}(a)
-e_{2i,2j-1}(a))\Bigr)
$$
for all $1 \leq i \neq j \leq n$ and
$$
h_{ij}(a,0) \mapsto \frac{1}{2}\Bigl(e_{2i-1,2j-1}(a)-e_{2i,2j}(a)+
\sqrt{-1} (e_{2i-1,2j}(a) +e_{2i,2j-1}(a))\Bigr),
$$
$$
g_{ij}(a,0) \mapsto \frac{1}{2}\Bigl(e_{2i-1,2j-1}(a)-e_{2i,2j}(a)-
\sqrt{-1} (e_{2i-1,2j}(a) +e_{2i,2j-1}(a))\Bigr)
$$
for all $1 \leq i, j \leq n$.

In the reverse direction, we can define $\psi$ : $\sl_{2n}(S)
\rightarrow \fg_{2n,\rho}(R)$ by
$$
e_{2i-1,2j-1}(a) \mapsto \frac{1}{2}(f_{ij}(a,0)-f_{ji}(0,a)+
h_{ij}(a,0)+g_{ij}(a,0))
$$
$$
e_{2i,2j}(a) \mapsto \frac{1}{2}(f_{ij}(a,0)-f_{ji}(0,a)-
h_{ij}(a,0)-g_{ij}(a,0))
$$
$$
e_{2i-1,2j-1}(a) \mapsto \frac{1}{2\sqrt{-1}}
(f_{ij}(a,0)-f_{ji}(0,a)+ h_{ij}(a,0)-g_{ij}(a,0))
$$
and
$$
e_{2i-1,2j-1}(a) \mapsto \frac{1}{2\sqrt{-1}}
(-f_{ij}(a,0)-f_{ji}(0,a)+ h_{ij}(a,0)-g_{ij}(a,0)).
$$
We get that $\varphi \circ \psi =\id_{\sl_{2n}(S)}$ and $\psi \circ
\varphi = \id_{\fg_{2n,\rho}(R)}$. Therefore $\varphi$ is an
isomorphism of Lie algebras.
\end{proof}

\begin{rem}
If $R$ is an associative algebra over $\C$, then for $n \geq 2$ we
have
$$
\fg_{2n,1}(R) \cong \eu_{2n}(R,\bar{\;\;\;}),
$$
where $\eu_{2n}(R,\bar{\;\;\;})$ is the elementary unitary Lie algebra
studied by Zheng, Chang and Gao in \cite{ZCG}.
Moreover, finite-dimensional irreducible representations of
the elementary unitary Lie algebras $\eu_{2n}(\C_q,\bar{\;\;\;})$ have been studied in \cite{CGZ}.
\end{rem}
\begin{proof}
The following map is an isomorphism between the two algebras:
\begin{equation*}
\begin{split}
\fij(a) = &\frac{1}{2} \Big( \xi_{2i-1,2j-1}(a) + \xi_{2i,2j}(a) + \sqrt{-1}
            \big( \xi_{2i-1,2j}(a) - \xi_{2i,2j-1}(a) \big)   \Big) \\
\gij(a) = &\frac{1}{2} \Big( \xi_{2i-1,2j-1}(a) - \xi_{2i,2j}(a) - \sqrt{-1}
            \big( \xi_{2i-1,2j}(a) + \xi_{2i,2j-1}(a) \big)   \Big) \\
\hij(a) = &\frac{1}{2} \Big( \xi_{2i-1,2j-1}(a) - \xi_{2i,2j}(a) + \sqrt{-1}
            \big( \xi_{2i-1,2j}(a) + \xi_{2i,2j-1}(a) \big)   \Big),
\end{split}
\end{equation*}
where
\begin{equation*}
\xi_{ij}(a) = - \xi_{ji}(\ovla) = e_{ij}(a) - e_{ji}(\ovla).  \qedhere
\end{equation*}
\end{proof}

Given an involutive ideal $J$ of $R$, that is, an ideal
closed under involution, and $n \geq 2$, consider the three
subspaces of $\fg_{2n,\rho}(R)$ defined as follows:
$$
\widetilde{\fg_{2n,\rho}}(J) =\left((\sum_{i=1}^nf_{ii})((R_+ +
[R,R])\cap \hat{J} \cap \widehat{J_+})\right) \oplus
\sum\limits_{i=2}^n(f_{ii}-f_{11})(J) \oplus \Delta(J),
$$
$$
\overline{\fg_{2n,\rho}}(J) =f_{11}((R_+ + [R,R])\cap J) \oplus
\sum\limits_{i=2}^n(f_{ii}-f_{11})(J) \oplus \Delta(J), \text{ and }
$$
$$
\underline{\fg_{2n,\rho}}(J) =f_{11}(J_+ + [J,R]) \oplus
\sum\limits_{i=2}^n(f_{ii}-f_{11})(J) \oplus \Delta(J),
$$
where
$$
\Delta(J)= \sum\limits_{i \neq j} f_{ij}(J) \oplus \sum\limits_{i
\leq j} \left(g_{ij}(J) \oplus h_{ij}(J)\right),
$$
$\hat{J}$ is the ideal $\{a \in R \,|\, [a,R] \subset J\}$ of Lie algebra $R$,
and $\widehat{J_+}=\{a \in R\,|\, a+\bar{a} \in J\}$.

\begin{rem}\label{rem:ovl=undrl}
If $J$ is an involutive ideal of $R$ then
$\underline{\fg_{2n,\rho}}(J) \subset \overline{\fg_{2n,\rho}}(J)
\subset \widetilde{\fg_{2n,\rho}}(J) $.
Moreover, if
$$J_+ \oplus ([R,R]\cap J_-)=(R_+ + [R,R])\cap J=J_+ + [J,R],
$$
that is, $[R,R]\cap J_- \subset [J,R]$, then
$$
\underline{\fg_{2n,\rho}}(J) = \overline{\fg_{2n,\rho}}(J).
$$
\end{rem}

\begin{lem}
For any involutive ideal $J$ of $R$,
$\widetilde{\fg_{2n,\rho}}(J)$, $\overline{\fg_{2n,\rho}}(J)$ and
$\underline{\fg_{2n,\rho}}(J)$ are all ideals of $\fg_{2n,\rho}(R)$.
\qed
\end{lem}

We study the quotient $\fg_{2n,\rho}(R)/\overline{\fg_{2n,\rho}}(J)$
in the following theorem:

\begin{pro}\label{thm:gtnrR/ovlgtnrJ cong gtnr(R/J)}
If $J$ is an involutive ideal of $R$, then the involution
$\,\bar{\;}\,$ on $R$ induces an involution on the quotient
$R/J$ in a natural way, which we again denote $\,\bar{\;}\,$.
For $n \geq 2$, we have
$\fg_{2n,\rho}(R)/\overline{\fg_{2n,\rho}}(J) \cong
\fg_{2n,\rho}(R/J)$.
\end{pro}
\begin{proof}
Let $\pi : \gtnrR \to \gtnr(R/J)$ be given by $(a_{ij}) \mapsto (a_{ij} + J) $.
$\pi$ is a surjective homomorphism.  Furthermore, $(a_{ij} + J) = (0+J)$
if and only if $a_{ij} \in J$ for all $i,j$.
That is, $\ker \pi = \gl_{2n}(J) \cap \gtnrR$.
By Corollary \ref{cor:gtnrR} and the definition of $\ovlgtnrJ$, this means that $\ker \pi = \ovlgtnrJ$.
\end{proof}

Conversely, given an ideal $\J$ in $\fg_{2n,\rho}(R)$,
the following theorem gives us an involutive ideal $J$ of $R$
corresponding to $\J$.

\begin{thm}\label{thm:ideal calJ gives ideal J}
Let $n \geq 3$ and $\J$ be an ideal of $\fg_{2n,\rho}(R)$.
\begin{enumerate}
\item The subsets
\begin{equation*}
\begin{split}
J_{i,j} = & \{a \in R \,|\, f_{ij}(a) \in \J\}, \quad  1 \leq i \neq j \leq n, \\
J_{i,n+j}= &\{a \in R \,|\, g_{ij}(a) \in \J\}, \quad  1 \leq i \neq j \leq n,  \\
J_{n+i,j}= &\{a \in R \,|\, h_{ij}(a) \in \J\}, \quad  1 \leq i \neq j \leq n.
\end{split}
\end{equation*}
are all involutive ideals of R and independent of the choice of $(i,j)$.
We denote them $J_1$, $J_2$ and $J_3$, respectively. Moreover, we have
$$
J=J_1=J_2=J_3.
$$
\item Let
\begin{equation*}
I_{ij}=\{a \in R \,|\, (f_{ii}-f_{jj})(a) \in \J\}, \quad  1 \leq i
\neq j \leq n.
\end{equation*}
Then
$$
J=I_{ij} \text{ for all } 1 \leq i \neq j \leq n.
$$
\item Let
$$
I_{i}=\{a \in R \,|\, f_{ii}(a) \in \J\}, \quad  1 \leq i \leq n.
$$
Then they are independent of $i$ and we denote the common set by
$I$. Furthermore,
$$
J_+ +[J,R] \  \subset \ I \ \subset \ (R_+ +[R,R]) \cap J.
$$
\item Let
$$
J_{i,n+i}=\{a \in R \,|\, g_{ii}(a) \in \J\}, J_{n+i,i}=\{a \in R
\,|\, h_{ii}(a) \in \J\}, \;\;  1 \leq i \leq n.
$$
Then
$$
\overline{J_{i,n+i}}=J_{i,n+i}, \qquad
\overline{J_{n+i,i}}=J_{n+i,i},
$$
and
$$
J \ \subset \ J_{i,n+i}, \ J_{n+i,i} \ \subset \ J_{-\rho} \oplus
R_\rho.
$$
Furthermore,
$$
g_{ii}(J_{i,n+i})=g_{ii}(J), \text{ and }
h_{ii}(J_{n+i,i})=h_{ii}(J) \text{ for all } 1 \leq i \leq n.
$$
\item $\J$ has the following vector space
decomposition
\begin{equation}\label{eqn:calJ decomp}
\J=\J_0 \oplus \sum_{i \neq j} \J_{\epsilon_i-\epsilon_j} \oplus
\sum_{i<j}(\J_{\epsilon_i+\epsilon_j} \oplus
\J_{-\epsilon_i-\epsilon_j})\oplus \sum_{i} (\J_{2\epsilon_i} \oplus
\J_{-2\epsilon_i})
\end{equation}
where
$$
\J_\alpha = \G_\alpha \cap \J \text{ for all } \alpha \in \Delta_C
$$
and
\begin{equation}\label{eqn:calJzero decomp}
\J_0=\left(((\sum_{i=1}^nf_{ii})(R_+ + [R,R])) \cap \J\right) \oplus
\left( \left(\sum\limits_{i=2}^n(f_{ii}-f_{11})(R) \right) \cap \J
\right).
\end{equation}
Moreover,
\begin{equation}\label{eqn:tildegtnrJ contains calJ contains undrlgtnrJ}
\widetilde{\fg_{2n,\rho}}(J) \ \supset \  \J \ \supset \
\underline{\fg_{2n,\rho}}(J).
\end{equation}
\end{enumerate}
\end{thm}
\begin{proof}
(1).  Fix $ 1 \leq i \neq j \leq n$ and $ 1 \leq k \neq l \leq n$.
Take $a \in \Jij$.
\begin{enumerate}
\item[(i)]
Because $\fij(a) \in \calJ$, then
\begin{equation*}
\left[ \left[ \fji(1), \ \left[ \fij(a), \fjp(1) \right] \right], \ \fpi(1) \right] = \fji(a) \in \calJ.
\end{equation*}
So $a \in \Jji$ and $\Jij \subset \Jji$.
\item[(ii)]
Suppose $l \notin \{i,j\}$.  Then we have
$[ \fij(a), \fjl(1)] = \fil(a) \in \calJ$.
So $a \in \Jil$ and $\Jij \subset \Jil$.
\end{enumerate}

The following table illustrates how $\Jij \subset \Jkl$ across the other
various possibilities for $k$ and $l$. 
\begin{center}
\begin{tabular}{|c|c|l|}
\hline
\multicolumn{3}{|c|}{Table 1} \\
\hline
$k$ & $l$ & \\
\hline
$j$ & $\notin\{i,j\}$ & $\Jij \subset \Jji \subset \Jjl$ by (i) and (ii), respectively \\
$\notin\{i,j\}$ & $i$ & $\Jij \subset \Jik \subset \Jki$ by (ii) and (i), respectively \\
$\notin\{i,j\}$ & $j$ & $\Jij \subset \Jji \subset \Jjk \subset \Jkj$ by (i), (ii) and (i), respectively \\
$\notin\{i,j\}$ & $\notin\{i,j\}$ & $\Jij \subset \Jik \subset \Jki \subset \Jkl$ by (ii), (i) and (ii), respectively \\
\hline
\end{tabular}
\end{center}
Thus, $\Jij = \Jkl$ for all $1 \leq i \neq j \leq n$
and $1 \leq k \neq l \leq n$.  Let $J_1$ denote this common set
$\Jij$.

Similarly, we have $\Jinpj = \Jknpl$ and $\Jnpij = \Jnpkl$ for all $1
\leq i \neq j \leq n$, $1 \leq k \neq l \leq n$ respectively.  Let $\Jtwo$
$\Jthree$ denote these common set $\Jinpj$ and $\Jnpij$ respectively.

Note that, for all $a \in R$, we have
\begin{equation*}
[f_{12}(a), g_{23}(1)] = g_{13}(a), \; \big[ [g_{13}(a), h_{32}(1)], \ h_{31}(-1)\big] =
h_{32}(a)
\end{equation*}
and
\begin{equation*}
[ g_{13}(1), h_{32}(a) ] = f_{12}(a).
\end{equation*}
Hence $\Jone \subset \Jtwo \subset \Jthree \subset \Jone$, i.e.,
$\Jone = \Jtwo = \Jthree$.  We denote this common set by $J$.

$J$ is, in fact, an involutive ideal of $R$.  Indeed, let $a,b \in
J$, $\alpha \in k$, and $r \in R$. Then
$f_{12}(a) + f_{12}(b) = f_{12}(a+b) \in \calJ$ and $\alpha
f_{12}(a) = f_{12}(\alpha a) \in \calJ$.  That is, $a+b$ and $\alpha
a$ are members of $J$.  Since $[f_{12}(a), f_{23}(r)] = f_{13}(ar)$ and
$[f_{31}(r), f_{12}(a)] = f_{32}(ra)$.  That is, $ar \in J$
and $ra \in J$.  Thus $J$ is an ideal of
$R$.  Furthermore, we have that $[f_{12}(a), g_{32}(1)] = g_{31}(\ovla) \in \calJ$.
That is, $\ovla \in J$.  So $J$ is closed under involution.

\vskip0.15in

(2) Fix $1 \leq i \neq j \leq n$.  Take $a \in I_{ij}$ and choose
$p \notin \{i,j\}$, then
\begin{equation*}
[ (\fii - \fjj)(a), \gip(1) ] = [ \fii(a), \gip(1)] - [\fjj(a), \gip(1)]
= \gip(a) - 0 = \gip(a) \in \calJ,
\end{equation*}
which implies that $ a \in J_{i,n+p} = J$.  So $I_{ij} \subset J$.

For the reverse inclusion, take $a \in J$, then
\begin{equation*}
[\fij(a), \fji(1)] = \fii(a) - \fjj(a) \in \calJ,
\end{equation*}
which implies that $a \in I_{ij}$.  That is, $J \subset I_{ij}$.

\vskip0.15in

(3) Let $1 \leq i \leq n$ and $a \in I_i$.  Then $\fii(a) \in
\calJ$.  Take any $1 \leq j \leq n$ with $j \neq i$, then
\begin{equation*}
[ \fji(1), \ [ \fii(a), \fij(1) ] ] = [ \fji(1), \fij(a) ] = \fjj(a) - \fii(a) \in \calJ.
\end{equation*}
But then $\fjj(a) = \left( \fjj(a) - \fii(a) \right) + \fii(a) \in \calJ$,
implying that $a \in I_j$.
That is, $I_i = I_j$.  We denote this common set by $I$.

Next, let us show that $\Jplus + \JR \subset I$.  Take $a \in
\Jplus$.  Then $\ovla = a$ and
\begin{equation*}
[ \gij(\ovla), \hji\left(\frac{1}{2}\right) ] + [ \fij(a), \fji\left(\frac{1}{2}\right)]  = \fii(a).
\end{equation*}
Hence $a \in I_i = I$ and $\Jplus \subset I$.  If $a \in J$ and $b
\in R$, then
\begin{equation*}
[ \fij(a), \fji(b)] - [\fij(ba), \fji(1)] = \fii(ab) - \fjj(ba) - \fii(ba) + \fjj(ba) = \fii([a,b]).
\end{equation*}
So $[a,b] \in I_i = I$ and $\JR \subset I$.  Thus $\Jplus + \JR \subset I$.

Next, we show that $I \subset \RplusRRandJ$.  Take $a \in I$.  Since
$f_{11}(a) \in \calJ \subset \gtnrR$.  But then $a$ must be a
member of $\RplusRR$.  Hence $I \subset \RplusRR$.  Moreover,
$f_{11}(a) \in \calJ$ implies that
\begin{equation*}
[ f_{11}(a), g_{12}(1)] = g_{12}(a) \in \calJ.
\end{equation*}
That is, $a \in J_{1, n+2} = J$.  So $I \subset J$.  Thus, $I
\subset \RplusRRandJ$.

\vskip0.15in

(4) Fix $1 \leq i \leq n$. Since $\gii(a) = -\rho\gii(\ovla)$,
thus $\Jinpi=\overline{\ovlJinpi} \subset \ovlJinpi \subset \Jinpi$.
That is $\ovlJinpi = \Jinpi$.
The proof that $\ovlJnpii = \Jnpii$ is similar.

Take $a \in J$. Choose $1 \leq j \leq n$ such that $j \neq i$.
Since
\begin{equation*}
[\fij(a), \gji(1)] = \gii(a) \in \calJ.
\end{equation*}
So $a \in \Jinpi$ and $J \subset \Jinpi$.

Next, we show that $\Jinpi \subset \Jmrho \oplus \Rrho $.
Take $a \in \Jinpimrho$.  Then $\gii(a) \in \calJ$ and $\ovla =
-\rho a$.  Choose an index $j$ between $1$ and $n$ such that $j \neq
i$. Then
\begin{align*}
[\fji(1), \gii(a)] = \gji(a) + \gij(a)  = \gij(-\rho \ovla) + \gij(a)
= \gij( \rho^2 a + a) = \gij(2a).
\end{align*}
So $2a$ and, hence, $a \in \Jinpj = J$.  That is, $\Jinpimrho
\subset J = \Jmrho \oplus \Jrho \subset \Jinpi$, it follows $\Jinpimrho = \Jmrho$.
But then $\Jinpi = \Jinpimrho \oplus \Jinpirho = \Jmrho \oplus
\Jinpirho \subset \Jmrho \oplus \Rrho$.

The proof that $J \subset \Jnpii \subset \Jmrho \oplus \Rrho $ is similar.

Next, we show that $\gii(\Jinpi) \subset \gii(J)$.
Take $\gii(a) \in \gii(\Jinpi)$.  So $a \in \Jinpi$.  Since $\Jinpi
\subset \Jmrho \oplus \Rrho$, we may write $a = \amrho + \arho$ for
some $\amrho \in \Jmrho$ and $\arho \in \Rrho$.  But then
\begin{equation*}
\begin{split}
\gii(a)  = & \einpi\left( (\amrho + \arho) - \rho( \overline{\amrho + \arho } ) \right) \\
= & \einpi( \amrho - \rho \ovlamrho ) = \gii(2\amrho) \in \gii(\Jmrho) \subset \gii(J).
\end{split}
\end{equation*}
Thus $\gii(\Jinpi) = \gii(J)$.
The proof that $\hii(\Jnpii) = \hii(J)$ is similar.

\vskip0.15in

(5). Using Proposition 1.5 in Kac [K], it is enough to show
 (\ref{eqn:calJzero decomp}). Assume
$u=\sum\limits_{i=1}^n f_{ii}(a_i) \in \J$. Then
$$
[u,f_{ij}(1)]=f_{ij}(a_i-a_j) \in \J \text{ for } 1 \leq i < j \leq
n.
$$
By part (1), we have $a_i-a_j \in J$. Set $a
=\frac{\sum\limits_{i=1}^n a_i}{n}$.  Then
$a_i-a=\frac{\sum\limits_{j=1}^n(a_i-a_j)}{n} \in J$.
Now, rewrite $u$ as follows
$$
u=(\sum_{i=1}^n f_{ii})(a)+\sum_{i=1}^nf_{ii}(a_i-a)=(\sum_{i=1}^n
f_{ii})(a)+\sum_{i=2}^n(f_{ii}-f_{11})(a_i-a).
$$
Hence $(\sum\limits_{i=1}^n f_{ii})(a) \in \J$, and
(\ref{eqn:calJzero decomp}) holds.

From parts (1)-(4), we first have $\J \supset\underline{\fg_{2n,\rho}}(J)$.
In order to prove (\ref{eqn:tildegtnrJ contains calJ
contains undrlgtnrJ}), it is enough to show
$$
(\sum_{i=1}^nf_{ii})((R_+ + [R,R])\cap \hat{J} \cap \widehat{J_+})
\supset ((\sum_{i=1}^nf_{ii})(R_+ + [R,R])) \cap \J.
$$
Let $u=(\sum\limits_{i=1}^nf_{ii})(a) \in \J$. For $b \in R$,
we have $[u,f_{12}(b)]=f_{12}([a,b]) \in \J$, which implies
$[a,b] \in J$ for all $b \in R$, i.e., $a \in \hat{J}$. On
the other hand, we have $[u,g_{12}(b)]=g_{12}(a+\bar{a}) \in \J$,
i.e., $a+\bar{a} \in J$ and $a \in \widehat{J_+}$.
\end{proof}

\begin{cor}
Let $R$ be an involutive simple associative algebra $($has no
nontrivial involutive ideal$)$ and $[R,R] \cap Z(R) \cap R_-=0$,
then $\fg_{2n,\rho}(R)$ is a simple Lie algebra for any $n \geq 3$.
\end{cor}
\begin{proof}
Let $\J$ be an ideal of $\fg_{2n,\rho}(R)$. $J$ is the ideal of $R$
corresponding to $\J$ . Since $R$ is involutive simple, we have $J =
0$ or $R$. If $J = 0$, then $\hat{J} = Z(R)$ and $\widehat{J_+} = R_-$, it
follows from (\ref{eqn:tildegtnrJ contains calJ contains
undrlgtnrJ}) that $\J = 0$. If $J = R$, we also have $\J =
\fg_{2n,\rho}(R)$.
\end{proof}

\begin{rem}
In Example \ref{eg:gtnrM}, since
$$
Z(M_m(k)) \cap [M_m(k),M_m(k)]=kI_m \cap [M_m(k),M_m(k)]=0
$$
and $M_m(k)$ is simple, hence without proof we can claim that $\fg_{2n,\rho}(M_m(k))$ are simple for $n \geq 3$.

In Example \ref{eg:gtnrSplusSop}, since $M_m(\C) \oplus M_m(\C)^{\rm{op}}$
is involutive simple and
$Z(M_m(\C) \oplus M_m(\C)^{\rm{op}}) \cap [M_m(\C)
\oplus M_m(\C)^{\rm{op}},M_m(\C) \oplus M_m(\C)^{\rm{op}}]=0$. Then
we have that $\fg_{2n,\rho}(M_m(\C) \oplus M_m(\C)^{\rm{op}})$ are simple
for $n \geq 3$.
\end{rem}

We end the section by giving more examples:

\begin{exa}\label{eg:gtnrCK}
Let $\K_4=\{1,\tau, \gamma,\tau \gamma\}$ be the Klein four group
and we denote it simply by $\K$. Let $\C \K$ be the group algebra of
$\K$ over $\C$ and the involution given by
\begin{equation}
\bar{1}=1, \; \bar{\tau}=-\tau \text{ and } \bar{\gamma}=-\gamma.
\end{equation}
Then we have
$$
\fg_{2n,\rho}(\C\K) \cong \sl_{2n}(\C) \oplus \sl_{2n}(\C).
$$
\end{exa}
\begin{proof}

Firstly, we have $\overline{\tau\gamma}=\bar{\gamma}\bar{\tau}
=\tau\gamma$ and
$$
(\C\K)_+=\C \oplus \C \tau\gamma, \quad (\C\K)_-=\C\tau \oplus \C
\gamma.
$$
Then all the involutive ideals of $\C\K$ are trivial ideals $0$,
$\C\K$ and
\begin{equation*}
J_1=\C(1+\tau\gamma) \oplus \C(\tau + \gamma), \; J_2=\C(1-\tau\gamma) \oplus \C(\tau - \gamma).
\end{equation*}

It is obvious that
$$
\C\K=J_1 \oplus J_2 \text{ and }\underline{\fg_{2n,\rho}}(J_i) =
\overline{\fg_{2n,\rho}}(J_i)=\fg_{2n,\rho}(J_i) \text{ for } i=1,2,
$$
where we consider $J_1$ and $J_2$ as associative algebras with
identity $\frac{1+\tau\gamma}{2}$ and $\frac{1-\tau\gamma}{2}$ respectively
and the involution induced by the involution on $\C\K$. Moreover, by
Lemma \ref{lemma:gtnrAoplusB equals gtnrAoplusgtnrB} in what follows
we have
$$
\fg_{2n,\rho}(\C\K)=\fg_{2n,\rho}(J_1) \oplus \fg_{2n,\rho}(J_2).
$$

Now by Example \ref{eg:gtnrSplusSop}, we just need show
$$
J_i \cong \C \oplus \C^{\rm{op}}=\C \oplus \C, \text{ for } i=1,2.
$$
In fact, the isomorphisms are given by
\begin{equation*}
\frac{1}{4}\bigl(1 \pm \tau\gamma+(\tau \pm \gamma)\bigr) \mapsto
(1,0) , \quad \frac{1}{4}\bigl(1 \pm \tau\gamma-(\tau \pm
\gamma)\bigr) \mapsto (0,1).  \qedhere
\end{equation*}
\end{proof}

\begin{exa}\label{eg:gtnrtildeM}
Define a new involution $\,\widetilde{\;}\,$ on $M_2(\C)$ given by
$$
\widetilde{A}=\left(
\begin{array}{cc}
0 & 1 \\
-1 & 0 \\
\end{array}
\right) A^t \left(
\begin{array}{cc}
0 & -1 \\
1 & 0 \\
\end{array}
\right).
$$
Now we denote $M_2(\C)$ with this involution by $\m2$. Then we have
$$
\fg_{2n,\rho}(\m2) \cong \fg_{4n,-\rho}(\C) \cong
\fg_{2n,-\rho}(M_2(\C)).
$$
\end{exa}
\begin{proof}
By Example \ref{eg:gtnrM}, we just need prove
$$
\fg_{2n,\rho}(\m2) \cong \fg_{2n,-\rho}(M_2(\C)).
$$

At first, we have $\m2$ is a simple associative algebra, $Z(\m2)=\C
I_2$ and $\widetilde{E_{ab}}=(-1)^{a+b}E_{3-b,3-a}$, therefore we
have
$$
\m2_+=\C I_2, \quad \m2_-=\sl_2(\C)
$$
and $\fg_{2n,\rho}(\m2)$ is a simple Lie algebra. Computing the
dimension of these Lie algebras we have
$$
\dim \fg_{2n,\rho}(\m2) = \dim \fg_{2n,-\rho}(M_2(\C))=8n^2 +2 \rho
n = \dim \fg_{4n,-\rho}(\C).
$$

Now we give an isomorphism $\varphi$ from $\fg_{2n,-\rho}(M_2(\C))$
to $\fg_{2n,\rho}(\m2)$:
$$
f_{ij}(E_{ab}) \mapsto  f_{ij}(E_{ab}) \text{ for all }  1 \leq i
\neq j \leq n
$$
and for all $1 \leq i, j \leq n$
\begin{equation*}
g_{ij}(E_{ab}) \mapsto (-1)^{3-b}g_{ij}(E_{a,3-b}), \;
h_{ij}(E_{ab}) \mapsto (-1)^{3-a}h_{ij}(E_{3-a,b}).  \qedhere
\end{equation*}
\end{proof}

\section{The quantum torus $\C_q$}

In this section, we discuss the properties of the quantum torus
$\C_q[t_1^{\pm 1}, t_2^{\pm 1}]$ as an associative algebra with involution. The main goal of
this section is the involutive ideals of $\C_q$.

The quantum torus $\C_q[t_1^{\pm 1}, t_2^{\pm 1}]$ is the
associative algebra over $\C$ generated by $t_1^{\pm 1}$, $t_2^{\pm
1}$ and subject to the relations $t_1t_1^{-1} = t_1^{-1}t_1 =
t_2t_2^{-1} = t_2^{-1}t_2 = 1$, $t_2t_1 = qt_1t_2$ (see \cite{M}).
We denote it simply by $\C_q$. It is easy to see that $t_1^it_2^j$,
$i, j \in \Z$ form a basis of $\C_q$ and $\C_q$ is a $\Z \times \Z$-graded algebra.
If $q = \pm 1$, then $\C_q$ is
called an elementary quantum torus (cf. \cite{Y2}).
Yoshii in his PhD thesis proved that

\begin{pro}[Proposition 3.2.7 in \cite{Y1}] \label{Yoji}
There exists a graded involution on $\C_q$ if and only if the quantum torus $\C_q$ is elementary.
In this case, $\tonebar = \thetaone \tone$ and $\ttwobar = \thetatwo \ttwo$ with $\theta_i = \pm 1$ for $i=1,2$.
\end{pro}

In particular, $\C_q$ is symmetric about $t_1$ and $t_2$.
So, instead of looking at all eight possible triples
$(q,\thetaone,\thetatwo)$, we can, without loss of generality,
restrict our attention to the parameters in the following six cases:
$$
(q,\theta_1,\theta_2) \in \big\{(-1,1,1), (-1,1,-1), (-1,-1,-1),
(1,1,1), (1,1,-1), (1,-1,-1) \big\}.
$$
We denote the set on the right by $\P$.

\begin{lem}[Proposition 2.44 in \cite{BGK}] \label{lemma:Cq's structure}
Let $q$ be a primitive root of unit of order $m$. Then we have
\begin{enumerate}
\item The center $Z(\C_q)$ of $\C_q$ has a basis consisting of
monomials $t_1^kt_2^l$ for $m|k$ and $m|l$.
\item The Lie algebra $[\C_q,\C_q]$ has a basis consisting of
monomials $t_1^kt_2^l$ for $m \nmid k$ or $m \nmid l$.
\item $\C_q=[\C_q,\C_q] \oplus Z(\C_q)$.
\end{enumerate}
\end{lem}

In the rest part of this paper, we always assume that $q= \pm 1$
and $m$ is the order of $q$. Note that $m=1$ or $2$.

\begin{defn}
For $\alpha, \beta \in \C^*$ and
$(q,\theta_1,\theta_2) \in \P$, we define the ideal
$J(\alpha,\beta)$ of $\C_q$ as follows:
\begin{enumerate}
\item If $q=-1$ or $(q,\theta_1,\theta_2)=(1,-1,-1)$, then
$J(\alpha,\beta)$ is generated by $(t_1-\alpha)(t_1+\alpha)$, $(t_2
-\beta)(t_2+\beta)$.
\item If $(q,\theta_1,\theta_2)=(1,1,1)$, then $J(\alpha,\beta)$ is
generated by $t_1-\alpha$, $t_2 -\beta$.
\item If $(q,\theta_1,\theta_2)=(1,1,-1)$, then $J(\alpha,\beta)$ is
generated by $t_1-\alpha$, $(t_2 -\beta)(t_2+\beta)$.
\end{enumerate}
\end{defn}

\begin{rem}
By Lemma \ref{lemma:Cq's structure}, for $\alpha,\beta \in \C^*$,
we know that the generators of the ideal
$J(\alpha,\beta)$ always lie in the center of $\C_q$. So each
element of $J(\alpha,\beta)$ can be written as a $\C_q-$linear
combination of the two generators.
\end{rem}

\begin{lem}
Let $\alpha,\beta \in \C^*$, then
$J(\alpha,\beta)$ is an involutive ideal.
\end{lem}
\begin{proof}
Let $x = (\tone - \alpha)(\tone + \alpha)$, $y = (\ttwo - \beta)(\ttwo + \beta)$,
$u = \tone - \alpha$, $v= \ttwo - \beta$, and $(q,\thetaone,\thetatwo) \in \P$.
Note that
\begin{equation*}
\begin{split}
\overline{x} = & \overline{(\tone - \alpha)(\tone + \alpha)} =
(\overline{\tone}+\alpha) (\overline{\tone}-\alpha) \\
= & (\thetaone\tone + \alpha)(\thetaone\tone-\alpha)
= (\tone + \thetaone \alpha)(\tone-\thetaone\alpha) = x.
\end{split}
\end{equation*}
Similarly, $\overline{y} = y$.
Also, $\overline{u} = \thetaone\tone - \alpha$ and
$\overline{v} = \thetatwo\ttwo - \beta$.
\begin{enumerate}
\item[(i)]
Suppose $q=-1$ or $(q,\thetaone,\thetatwo) =(1,-1,-1)$.
Since a typical element of $\Jab$ is of the form $ax + by$
for some $a,b \in \Cq$.  But then
\begin{equation*}
\overline{ax+by} = \overline{x}\ \overline{a}+\overline{y}\ \overline{b}
= x\overline{a}+y\overline{b} = \overline{a}x+\overline{b}y \in \Jab.
\end{equation*}
\item[(ii)]
Suppose $(q,\thetaone,\thetatwo) = (1,1,1)$.
Take a typical element $au + bv$ in $\Jab$, where $a,b \in \Cq$.
Note that $\overline{u} = u$ and $\overline{v} = v$ since
$\thetaone = \thetatwo = 1$.
Then
\begin{equation*}
\overline{au+bv} = \overline{u}\ \overline{a} + \overline{v}\ \overline{b}
= u \overline{a}+v \overline{b} = \overline{a} u+\overline{b} v \in \Jab.
\end{equation*}
\item[(iii)]
Suppose $(q,\thetaone,\thetatwo) = (1,1,-1)$.
Take a typical element $au + by$ in $\Jab$, where $a,b \in \Cq$.
Note that $\overline{u} = u$ since $\thetaone = 1$.  Then
\begin{equation*}
\overline{au+by} = \overline{u}\ \overline{a} + \overline{y}\ \overline{b}
= u\overline{a}+y\overline{b} = \overline{a}u+\overline{b}y \in \Jab.
\end{equation*}
\end{enumerate}
So, across all the possibilities for $(q,\thetaone,\thetatwo)$, $\Jab$ is closed under involution.
\end{proof}

\begin{lem}\label{lemma:Cq/Jalphabeta}
For $\alpha,\beta \in \C^*$ and
$(q,\theta_1,\theta_2) \in \P$, we have
\begin{enumerate}
\item If $(q,\theta_1,\theta_2)=(-1,1,1)$ or $(-1,1,-1)$, there is an
isomorphism
$$
\C_q/J(\alpha,\beta) \cong M_2(\C)
$$
preserving the involutions.
\item If $(q,\theta_1,\theta_2)=(-1,-1,-1)$, there is an
isomorphism
$$
\C_q/J(\alpha,\beta) \cong \m2
$$
preserving the involutions, where the involution of $\m2$ is defined
as in Example 2.5.
\item If $(q,\theta_1,\theta_2)=(1,1,1)$, there is an isomorphism
$$
\C_q/J(\alpha,\beta) \cong \C
$$
preserving the involutions, where the involution of $\C$ is the
identity map.
\item If $(q,\theta_1,\theta_2)=(1,1,-1)$, there is an isomorphism
$$
\C_q/J(\alpha,\beta) \cong \C \oplus \C^{\rm{op}}
$$
preserving the involutions, where the involution of $\C \oplus
\C^{\rm{op}}$ is defined as in Example \ref{eg:gtnrSplusSop}.
\item If $(q,\theta_1,\theta_2)=(1,-1,-1)$, there is an isomorphism
$$
\C_q/J(\alpha,\beta) \cong \C \K
$$
preserving the involutions, where the involution of $\C \K$ is
defined as in Example \ref{eg:gtnrCK}.
\end{enumerate}
\end{lem}
\begin{proof}
\begin{enumerate}
\item
Define an algebra homomorphism $\varphi: \Cq \to \MtwoC$ by
$\varphi(\tone) = \left[ \begin{array}{rr} \alpha & 0 \\ 0 &
-\alpha \end{array} \right]$ and $\varphi(\ttwo) =
\left[ \begin{array}{rr} 0 & \sqrt{\thetatwo} \beta \\
\frac{1}{\sqrt{\thetatwo}}\beta & 0 \end{array} \right]$.
Since
\begin{equation*}
\begin{split}
\varphi(\ttwo)\varphi(\tone)
= & \left[ \begin{array}{rr} 0 & \sqrt{\thetatwo}\beta \\
\frac{1}{\sqrt{\thetatwo}}\beta & 0  \end{array} \right]
= \left[ \begin{array}{rr} 0 & -\sqrt{\thetatwo}\beta\alpha \\
\frac{1}{\sqrt{\thetatwo}}\beta\alpha & 0 \\ \end{array} \right] \\
= & \left[ \begin{array}{rr} -\alpha & 0 \\ 0 & \alpha \end{array} \right]
\left[ \begin{array}{rr} 0 & \sqrt{\thetatwo}\beta \\
\frac{1}{\sqrt{\thetatwo}}\beta & 0 \\ \end{array} \right] = q \varphi(\tone)\varphi(\ttwo),
\end{split}
\end{equation*}
we see that $\varphi$ respects the relation $\ttwo \tone = q \tone \ttwo$.
$\varphi$ also preserves the involution since
\begin{equation*}
\varphi(\tonebar) = \varphi(\thetaone \tone) = \varphi(\tone) =
\left[ \begin{array}{rr} \alpha & 0 \\
0 & -\alpha \end{array} \right] =
\left[ \begin{array}{rr} \alpha & 0 \\ 0 & -\alpha  \end{array} \right]^t
= \overline{\varphi(\tone)},
\end{equation*}
and
\begin{equation*}
\begin{split}
\varphi(\ttwobar)
= & \varphi(\thetatwo \ttwo) =
\left[ \begin{array}{rr} 0 & \thetatwo\sqrt{\thetatwo}\beta \\
\thetatwo \frac{1}{\sqrt{\thetatwo}}\beta & 0 \end{array} \right] \\
= & \left[ \begin{array}{rr} 0 & \frac{1}{\sqrt{\thetatwo}}\beta \\
\sqrt{\thetatwo}\beta & 0 \end{array} \right]
= \left[ \begin{array}{rr} 0 & \sqrt{\thetatwo}\beta \\
\frac{1}{\sqrt{\thetatwo}}\beta & 0 \end{array} \right]^t
= \overline{\varphi(\ttwo)}.
\end{split}
\end{equation*}
Since
$\varphi(\tone^2 - \alpha^2) = 0$ and
$\varphi(\ttwo^2 - \beta^2) =0$, $\Jab \subset \ker \varphi$.
So $\varphi$ induces an algebra homomorphism
$\varphitilde: \Cq/\Jab \to \MtwoC$.
Since $\varphitilde(1 + \Jab)$, $\varphitilde(\tone + \Jab)$,
$\varphitilde(\ttwo + \Jab)$ and $\varphitilde(\tone\ttwo + \Jab)$
are linearly independent in $\MtwoC$, $\rank \varphitilde = 4$.
That is, $\varphitilde$ is surjective.
Since $\dim\left( \Cq/\Jab \right) = 4$,
$\varphitilde$ is also injective.  Hence,
$\varphitilde$ is an isomorphism and $\Cq/\Jab \cong \MtwoC$.

Similarly we can prove (2)-(5) with the following homomorphisms:
\item
$\varphi: \Cq \to \tildeMtwoC$ by $\tone \mapsto
\left[ \begin{array}{rr} \alpha & 0 \\
0 & -\alpha \end{array} \right] $,
$\ttwo  \mapsto
\left[ \begin{array}{rr} 0 & \beta \\
\beta & 0 \end{array} \right]$.
\item
$\varphi: \Cq \to \C$ by $\tone \mapsto \alpha$ and $\ttwo  \mapsto  \beta$.
\item
$\varphi: \Cq \to \CplusCop$ by
$\tone \mapsto (\alpha,\alpha)$ and $\ttwo  \mapsto  (\beta, -\beta)$.
\item
$\varphi: \Cq \to \CK$ by
$\tone \mapsto \alpha \tau$ and $\ttwo  \mapsto  \beta \gamma$.
\qedhere
\end{enumerate}
\end{proof}

\begin{defn}
Let $(q,\theta_1,\theta_2) \in \P$, and $f(T), g(T) \in \C[T]$ with
$f(0)g(0) \neq 0$ and $\deg f \deg g > 0$.  Then $J(f, g)$ is the
ideal of $\C_q$ generated by $f(t_1^m)$ and $g(t_2^m)$, where $m$ is
the order of $q$.
\end{defn}

\begin{lem} \label{lemma:Jfg closed under involution}
Let $(q,\theta_1,\theta_2) \in \P$ and $f(T), g(T) \in \C[T]$ with
$f(0)g(0) \neq 0$ and $\deg f \deg g > 0$. If $f(\theta_1^m T)=
f(T)$ and $g(\theta_2^m T)=  g(T)$,
then $J(f, g)$ is invariant under the involution
$\bar{\,\;}$.
\end{lem}
\begin{proof}
First we have
\begin{equation*}
\overline{f(\tone^m)} = f \left( \overline{\tone^m} \right) =
f\left( ( \overline{\tone} )^m \right) = f\left( (\thetaone\tone)^m \right)
= f\left( \thetaone^m \tone^m \right) = f(\tone^m),
\end{equation*}
where the last equality follows by hypothesis.  Similarly
$\overline{g(\ttwo^m)} = g(\ttwo^m)$.  Since $f(t_1^m)$ and $g(t_2^m)$
are in the center of $\Cq$, a typical element of $\Jfg$ is of the form
$a f(\tone^m) + b g(\ttwo^m)$ for some $a,b \in \Cq$.  But then
\begin{equation*}
\overline{ a f(\tone^m) + b g(\ttwo^m) } =
\overline{f(\tone^m)} \ovla + \overline{g(\ttwo^m)} \ovlb
= f(\tone^m) \ovla + g(\ttwo^m) \ovlb
= \ovla f(\tone^m) + \ovlb g(\ttwo^m)
\end{equation*}
and $\ovla f(\tone^m) + \ovlb g(\ttwo^m)$ lies in $\Jfg$.
That is, $\Jfg$ is closed under involution.
\end{proof}

\begin{rem}\label{remark:meaning of f(thetamT)=f(T)}
For a polynomial $f(T) \in \C[T]$ with $\deg f > 0$ and $f(0) \neq
0$, the roots of $f$ are all nonzero. The condition $ f(\eta^m T)=
 f(T)$ always holds if $\eta=1$ or $\eta=-1$, $m=2$. For $\eta=-1$ and $m=1$
the condition is equivalent to that if $\alpha$ is a root of $f$,
then $-\alpha$ is also a root of $f$ with the same multiplicity,
i.e., $f(T)$ can be written as
\begin{equation}
\lambda (T -\alpha_1)^{k_1} (T +\alpha_1)^{k_1} \cdots (T
-\alpha_s)^{k_s} (T + \alpha_s)^{k_s}
\end{equation}
where $\lambda \neq 0$, $s$ is a positive integer, $\pm \alpha_1,
\cdots, \pm \alpha_s \neq 0$ and distinct, $k_i > 0$, $i=1,\cdots,
s$.

Conversely, any polynomial $f$ of form (3.1) satisfies $ f(-T)=
f(T)$.
\end{rem}

\begin{rem}\label{remark:involution on AoplusB}
Let $\A$ and $\B$ be two associative algebras with involutions.
The direct sum $\A \oplus B$ is an associative algebra.
A new involution $\,\bar{\;}\,$ can be defined
by $\overline{(a, b)} = (\bar{a},\bar{b})$, for $a \in \A$, $b \in
\B$. $\A \oplus \B$ becomes an associative algebra with involution.
The involutions in the following lemma always have this meaning.
\end{rem}

In the rest part of the paper, if $f, g \in \C[T]$, the greatest
common divisor of $f$ and $g$ is denoted by $(f, g)$ and the formal
derivative of $f$ is denoted by $f'$. $f|g$ means that $f$ is a
divisor of $g$.

\begin{lem}\label{lemma:splitting Cq/Jf1f2g}
Let $(q,\theta_1,\theta_2) \in \P$ and $f_1, f_2, g_1,g_2 \in \C[T]$
with $(f_1,f_2)=1$, $(g_1,g_2)=1$, $f_1\cdot f_2\cdot g_1\cdot g_2(0) \neq 0$, $\deg f_1 \deg f_2 \deg g_1
\deg g_2 >0$ and
\begin{equation*}
f_i(\theta_1^m T)=  f_i(T), \; g_i(\theta_2^m T)=  g_i(T)
\end{equation*}
for $i=1,2$. Then there are isomorphisms of associative algebras
preserving the involutions
$$
\C_q/J(f_1 f_2,g_1) \cong \C_q/J(f_1,g_1) \oplus \C_q/J(f_2,g_1)
$$
and
$$
\C_q/J(f_1,g_1 g_2) \cong \C_q/J(f_1,g_1) \oplus \C_q/J(f_1,g_2).
$$
\end{lem}
\begin{proof}
Let $f = \fone\ftwo.$  The hypotheses, in combination with Lemma
\ref{lemma:Jfg closed under involution}, tell us that $\Jfgone$,
$\Jfonegone$, and $\Jftwogone$ are involutive ideals of $\Cq$.  The
involution on $\Cq$ induces involutions on the quotient algebras
$\Cq/\Jfgone$, $\Cq/\Jfonegone$, and $\Cq/\Jftwogone$.  The map
$\varphi: \Cq \to \Cq/\Jfonegone \oplus \Cq/\Jftwogone$ given by
\begin{equation*}
\varphi(a) = \big( a + \Jfonegone, \ a+ \Jftwogone \big)
\end{equation*}
is an associative algebra homomorphism that preserves the involution
on $\Cq$.  Moreover, it is surjective.  Indeed take any
$(a+\Jfonegone, \ b+\Jftwogone) \in \Cq/\Jfonegone \oplus
\Cq/\Jftwogone$, where $a, b \in \Cq$.  Since $(\fone, \ftwo) = 1$,
there exist $\hone, \htwo \in \C[T]$ such that
\begin{equation}\label{eqn:h1f1 plus h2f2 equals 1}
\hone\fone + \htwo\ftwo = 1.
\end{equation}
Let
\begin{equation*}
\aone = 1 - \hone(t_1^m)\fone(t_1^m) = \htwo(t_1^m)\ftwo(t_1^m),
\end{equation*}
and
\begin{equation*}
\atwo = \hone(t_1^m)\fone(t_1^m) = 1 - \htwo(t_1^m)\ftwo(t_1^m).
\end{equation*}
Then
\begin{equation*}
\varphi(\aone) = \big( 1 + \Jfonegone, \ 0 + \Jftwogone \big),
\quad \varphi(\atwo) = \big( 0 + \Jfonegone, \ 1 + \Jftwogone \big)
\end{equation*}
and
\begin{equation*}
\varphi(a\aone + b \atwo) = \big( a + \Jfonegone, \ b + \Jftwogone \big).
\end{equation*}
Since the kernel of this homomorphism $\varphi$ is equal to
$\Jfonegone \cap \Jftwogone$,
\begin{equation}\label{eqn:first iso thm}
\Cq/\Jfonegone \cap \Jftwogone \ \cong \ \Cq/\Jfonegone \oplus \Cq/\Jftwogone.
\end{equation}
We now show that $\Jfonegone \cap \Jftwogone = \Jfgone$.  Since $f =
\fone\ftwo$ and $\fone(t_1^m)$, $\ftwo(t_1^m)$ lie in the centre of
$\Cq$, it follows that $\Jfgone \subset \Jfonegone \cap \Jftwogone$.
To show the reverse inclusion, let $a \in \Jfonegone \cap
\Jftwogone$.  So there exist $\bone$, $\btwo$, $\cone$, and $\ctwo
\in \Cq$ such that
\begin{equation*}
a = \bone\fone(t_1^m) + \cone\gone(t_2^m) = \btwo\ftwo(t_1^m) + \ctwo\gone(t_2^m).
\end{equation*}
But then, by (\ref{eqn:h1f1 plus h2f2 equals 1}),
\begin{align*}
a & = a \big( \hone(t_1^m) \fone(t_1^m) + \htwo(t_1^m) \ftwo(t_1^m) \big) \\
& = \big( \btwo\ftwo(t_1^m) + \ctwo\gone(t_2^m) \big) \ \hone(t_1^m) \fone(t_1^m)
+ \big( \bone\fone(t_1^m) + \cone\gone(t_2^m) \big) \ \htwo(t_1^m) \ftwo(t_1^m) \\
& = \big( \bone \htwo(t_1^m) + \btwo \hone(t_1^m) \big) f(t_1^m)
+ \big( \ctwo \hone(t_1^m) \fone(t_1^m) + \cone \htwo(t_1^m) \ftwo(t_1^m) \big) \gone(t_2^m)
\end{align*}
which is an element of $\Jfgone$.
Hence $\Jfonegone \cap \Jftwogone = \Jfgone$ and, by (\ref{eqn:first
iso thm}),
\begin{equation*}
\Cq/\Jfgone \ \cong \ \Cq/\Jfonegone \oplus \Cq/\Jftwogone.
\end{equation*}
The proof that $\Cq/J(\fone, \gone\gtwo) \ \cong \ \Cq/\Jfonegone
\oplus \Cq/J(\fone,\gtwo)$ follows similarly.
\end{proof}

\section{Finite dimensional quotients of $\fg_{2n,\rho}(\C_q)$}

Suppose $\J$ is an ideal of $\fg_{2n,\rho}(\C_q)$ such that
$\fg_{2n,\rho}(\C_q)/\J$ is finite-dimensional and semisimple.
$\fg_{2n,\rho}(\C_q)/\J$ are determined in this section. The
classification of finite-dimensional quotients of
$\fg_{2n,\rho}(\C_q)$ is the premise of the classification of finite
dimensional irreducible representations of $\fg_{2n,\rho}(\C_q)$.

\begin{lem}\label{lemma:ovlgtnrJfg=undrlgtnrJfg}
Let $f(T), g(T) \in \C[T]$ such that $J(f,g)$ is involutive ideal of
$\C_q$, then for $n \geq 2$ we have
$$
\overline{\fg_{2n,\rho}}(J(f,g)) = \underline{\fg_{2n,\rho}}
(J(f,g)).
$$
\end{lem}
\begin{proof}
We take advantage of Remark \ref{rem:ovl=undrl} which says that it
suffices to prove the inclusion relation $[\Cq,\Cq] \cap \Jfg_{-} \subset [\Cq,\Jfg]$. Let
$u$ be an arbitrary element of $[\Cq,\Cq] \cap \Jfg_{-}$.  We may write it as
\begin{equation*}
u = a f(\tone^m) + b g(\ttwo^m)
\end{equation*}
for some $a, b \in \Cq$.  Since $\Cq = [\Cq,\Cq] \oplus Z(\Cq)$, we
can express $a$ as $a = \tilde{a} + c$ and $b$ as $b = \tilde{b} +
d$ uniquely for some $\tilde{a} = \sum\limits_{i=1}^k [a_i, a_i']$ and
$\tilde{b} = \sum\limits_{j=1}^l [b_j,b_j']$ in $[\Cq,\Cq]$, with $a_i,
a_i', b_j, b_j' \in \Cq$, and $c, d \in Z(\Cq)$.  But then
\begin{align*}
u & = a f(\tone^m) + b g(\ttwo^m) \\
& = (\tilde{a} + c) f(\tone^m) + (\tilde{b} + d) g(\ttwo^m)  = \tilde{a} f(\tone^m) + \tilde{b} g(\ttwo^m) + c f(\tone^m) + d g(\ttwo^m) \\
& = \Big( \sum_{i=1}^k [a_i, a_i'] \Big) f(\tone^m) + \Big( \sum_{j=1}^l [b_j,b_j']  g(\ttwo^m) \Big) + c f(\tone^m) + d g(\ttwo^m) \\
& = \sum_{i=1}^k [a_i, a_i' f(\tone^m) ] + \sum_{j=1}^l [b_j,b_j' g(\ttwo^m) ]  + c f(\tone^m) + d g(\ttwo^m),
\end{align*}
with the last equality following from $f(\tone^m)$ and $g(\ttwo^m)$
being members of $Z(\Cq)$.  Since $u \in [\Cq, \Cq]$ and $\Cq =
[\Cq,\Cq] \oplus Z(\Cq)$, the $Z(\Cq)$-component of $u$, that is, $c
f(\tone^m) + d g(\ttwo^m)$, must equal $0$.  So
\begin{equation*}
u = \sum_{i=1}^k [a_i, a_i' f(\tone^m) ] + \sum_{j=1}^l [b_j,b_j' g(\ttwo^m) ] \in [\Cq,\Jfg]. \qedhere
\end{equation*}
\end{proof}

\begin{lem}\label{lemma:gtnrAoplusB equals gtnrAoplusgtnrB}
Let $\A$ and $\B$ be two associative algebras with involutions and
$\A \oplus \B$ be the associative algebra with involution introduced
in Remark \ref{remark:involution on AoplusB}. Then we have
$$
\fg_{2n,\rho}(\A \oplus \B) \cong \fg_{2n,\rho}(\A) \oplus
\fg_{2n,\rho}(\B).
$$
\end{lem}
\begin{proof}
The map $\varphi: \mathfrak{g}_{2n,\rho}(\A \oplus \B) \to
\mathfrak{g}_{2n,\rho}(\A) \oplus \mathfrak{g}_{2n,\rho}(\B)$ given
by
\begin{align*}
& \fij\big( (a,b) \big) \mapsto \big( \fij(a), 0 \big) + \big( 0, \fij(b) \big) \qquad \text{for } i \neq j, \\
& \gij\big( (a,b) \big) \mapsto \big( \gij(a), 0 \big) + \big( 0, \gij(b) \big) \qquad \text{for } i \leq j, \text{ and} \\
& \hij\big( (a,b) \big) \mapsto \big( \hij(a), 0 \big) + \big( 0, \hij(b) \big) \qquad \text{for } i \leq j
\end{align*}
is an isomorphism.
\end{proof}

\begin{lem}\label{lemma:gtnr(Cq/Jfg)}
Let $(q,\theta_1,\theta_2) \in \P$ and $f(T), g(T) \in \C[T]$ with
$\deg f \deg g > 0$, $f(0)g(0) \neq 0$, $(f, f') = 1$, $(g, g') = 1$
and $f,g$ satisfy
\begin{equation*}
f(\theta_1^m T)=  f(T), \; g(\theta_2^m T)=  g(T).
\end{equation*}
Then for $n \geq 2$ we have
\begin{enumerate}
\item If $(q,\theta_1,\theta_2)=(-1,1,1)$ or $(-1,1,-1)$, there is an
isomorphism
$$
\fg_{2n,\rho}(\C_q/J(f, g)) \cong \fg_{4n,\rho}(\C)^{rs}
$$
where $r$ and $s$ are the degree of $f$ and $g$, respectively.
\item If $(q,\theta_1,\theta_2)=(-1,-1,-1)$, there is an
isomorphism
$$
\fg_{2n,\rho}(\C_q/J(f, g)) \cong \fg_{4n,-\rho}(\C)^{rs}
$$
where $r$ and $s$ are the degree of $f$ and $g$, respectively.
\item If $(q,\theta_1,\theta_2)=(1,1,1)$, there is an isomorphism
$$
\fg_{2n,\rho}(\C_q/J(f, g)) \cong \fg_{2n,\rho}(\C)^{rs}
$$
where $r$ and $s$ are the degree of $f$ and $g$, respectively.
\item If $(q,\theta_1,\theta_2)=(1,1,-1)$, there is an isomorphism
$$
\fg_{2n,\rho}(\C_q/J(f, g)) \cong \sl_{2n}(\C)^{rs}
$$
where $r$ and $2s$ are the degree of $f$ and $g$, respectively.
\item If $(q,\theta_1,\theta_2)=(1,-1,-1)$, there is an isomorphism
$$
\fg_{2n,\rho}(\C_q/J(f, g)) \cong \sl_{2n}(\C)^{2rs}
$$
where $2r$ and $2s$ are the degree of $f$ and $g$, respectively.
\end{enumerate}
\end{lem}
\begin{proof}
\begin{itemize}
\item[(1)]
Let $(q,\theta_1,\theta_2) = (-1,1,1) \text{ or } (-1,1,-1)$.
The conditions $f(0)g(0)\neq 0$, $\deg f \deg g >0$, and
$(f,f') = (g,g') = 1$ imply that
$f(T) = \lambda_1 (T- \alpha_1) \cdots (T-\alpha_r)$
and
$g(T) = \lambda_2 (T- \beta_1) \cdots (T- \beta_s)$
for some $\lambda_1, \lambda_2 \in \bbC^*$, $r, s \geq 1$,
distinct $\alpha_1, \ldots, \alpha_r \in \bbC^*$,
and distinct $\beta_1, \ldots, \beta_s \in \bbC^*$.

Recall that $\Jfg$ is generated by
\begin{equation*}
\begin{split}
f(\tone^2) = & \lambda_1 (\tone^2 - \alpha_1) \cdots (\tone^2 - \alpha_r) \\
= & \lambda_1 (\tone - \sqrt{\alpha_1} ) ( \tone + \sqrt{\alpha_1} ) \cdots
( \tone - \sqrt{\alpha_r} ) ( \tone + \sqrt{\alpha_r} ),
\end{split}
\end{equation*}
and
\begin{equation*}
\begin{split}
g(\ttwo^2) = & \lambda_2 (\ttwo^2 - \beta_1) \cdots (\ttwo^2 - \beta_s) \\
= & \lambda_2 ( \ttwo - \sqrt{\beta_1} ) ( \ttwo + \sqrt{\beta_1} ) \cdots
( \ttwo - \sqrt{\beta_s} ) ( \ttwo + \sqrt{\beta_s} ).
\end{split}
\end{equation*}
So, by Lemma \ref{lemma:splitting Cq/Jf1f2g} and Lemma \ref{lemma:Cq/Jalphabeta}(1), we have
\begin{align*}
\Cq/\Jfg \cong & \bigoplus_{i=1}^r \bigoplus_{j=1}^s \
\frac{\Cq}{J\left( (\tone - \sqrt{\alpha_i}) (\tone + \sqrt{\alpha_i}),
\ (\ttwo - \sqrt{\beta_j}) (\ttwo + \sqrt{\beta_j}) \right)}, \\ 
\cong & \bigoplus_{i=1}^r \bigoplus_{j=1}^s \ \frac{\Cq}{J( \sqrt{\alpha_i},
\ \sqrt{\beta_j} )}, \text{ by defn. of $J( \sqrt{\alpha_i}, \ \sqrt{\beta_j} )$} \\
\cong & \bigoplus_{k=1}^{rs}\  M_2(\bbC). 
\end{align*}
Hence Lemma \ref{lemma:gtnrAoplusB equals gtnrAoplusgtnrB} and Example \ref{eg:gtnrM} imply
\begin{equation*}
\begin{split}
\mathfrak{g}_{2n,\rho} \left( \Cq/\Jfg \right)
\cong  & \mathfrak{g}_{2n,\rho} \left( \bigoplus_{k=1}^{rs} M_2(\bbC) \right)
\cong \bigoplus_{k=1}^{rs} \ \mathfrak{g}_{2n,\rho} \left( M_2(\bbC) \right)  \\
\cong & \bigoplus_{k=1}^{rs} \ \mathfrak{g}_{4n,\rho} (\bbC) = \mathfrak{g}_{4n,\rho} (\bbC)^{\,rs}.
\end{split}
\end{equation*}

\item[(2)]
Let $(q,\theta_1,\theta_2) = (-1,-1,-1)$.
The argument follows exactly the same as in part (1) except that now
$\Cq/\Jfg \cong \tildeMtwoC$, by Lemma \ref{lemma:Cq/Jalphabeta} (2).
So $$\Cq/\Jfg \cong \bigoplus_{k=1}^{rs} \ \tildeMtwoC$$
and by Lemma \ref{lemma:gtnrAoplusB equals gtnrAoplusgtnrB} and
Example \ref{eg:gtnrtildeM}
\begin{equation*}
\begin{split}
\mathfrak{g}_{2n,\rho} \left( \Cq/\Jfg \right)
\cong & \mathfrak{g}_{2n,\rho} \left( \bigoplus_{k=1}^{rs}\  \tildeMtwoC \right)
\cong \bigoplus_{k=1}^{rs} \ \mathfrak{g}_{2n,\rho} (\tildeMtwoC ) \\
\cong & \bigoplus_{k=1}^{rs} \ \mathfrak{g}_{4n,-\rho} (\bbC) = \mathfrak{g}_{4n,-\rho} (\bbC)^{\,rs}.
\end{split}
\end{equation*}

\item[(3)]
Let $(q,\theta_1,\theta_2) = (1,1,1)$.  The conditions $\deg f \deg g >0$,
$f(0)g(0)\neq 0$, and $(f,f') = (g,g') = 1$ imply that
$f(T) = \lambda_1 (T- \alpha_1) \cdots (T-\alpha_r)$
and
$g(T) = \lambda_2 (T- \beta_1) \cdots (T- \beta_s)$ for some
$\lambda_1, \lambda_2 \in \bbC^*$, $r, s \geq 1$, distinct
$\alpha_1, \ldots, \alpha_r \in \bbC^*$, and distinct
$\beta_1, \ldots, \beta_s \in \bbC^*$.
The ideal $\Jfg$ is generated by
$f(\tone)= \lambda_1 (\tone- \alpha_1) \cdots (\tone-\alpha_r)$
and
$g(\ttwo) = \lambda_2 (\ttwo- \beta_1) \cdots (\ttwo- \beta_s)$.
Hence by Lemma \ref{lemma:splitting Cq/Jf1f2g} we have
\begin{equation*}
\Cq/\Jfg \cong \bigoplus_{i=1}^r \bigoplus_{j=1}^s
\ \frac{\Cq}{J\left( \tone - \alpha_i, \ \ttwo - \beta_j \right)}
\cong \bigoplus_{i=1}^r \bigoplus_{j=1}^s \ \bbC
= \bigoplus_{k=1}^{rs} \  \bbC
\end{equation*}
and by Lemma \ref{lemma:gtnrAoplusB equals gtnrAoplusgtnrB}
\begin{equation*}
\mathfrak{g}_{2n,\rho} \left( \Cq/\Jfg \right)
\cong \mathfrak{g}_{2n,\rho} \left( \bigoplus_{k=1}^{rs} \bbC \right)
\cong \bigoplus_{k=1}^{rs} \ \mathfrak{g}_{2n,\rho} \left( \bbC \right)
= \mathfrak{g}_{2n,\rho} (\bbC)^{\,rs}.
\end{equation*}

\item[(4)]
Let $(q,\theta_1,\theta_2) = (1,1,-1)$.
We begin as in part (3) except the condition
$g(\thetatwo^m T) = g\left( (-1)^1 (T) \right) = g(-T)$
imposes the constraint that if $\beta_j$ is a root of $g$,
then so is $-\beta_j$.
So $f(T) = \lambda_1 (T- \alpha_1) \cdots (T-\alpha_r)$ and
$g(T) = \lambda_2 (T- \beta_1)(T + \beta_1) \cdots (T - \beta_s)(T + \beta_s)$
for some $\lambda_1, \lambda_2 \in \bbC^*$, $r, s \geq 1$, distinct
$\alpha_1, \ldots, \alpha_r \in \bbC^*$, and distinct $\beta_1, \ldots, \beta_s \in \bbC^*$.

Since $\Jfg$ is generated by $f(\tone)$ and $g(\ttwo)$,
\begin{align*}
\Cq/\Jfg & \cong \bigoplus_{i=1}^r \bigoplus_{j=1}^s
\ \frac{\Cq}{J\left( \tone - \alpha_i, \ (\ttwo - \beta_j) (\ttwo + \beta_j) \right)} \\
& \cong \bigoplus_{i=1}^r \bigoplus_{j=1}^s \ \frac{\Cq}{J( \alpha_i, \ \beta_j )}
 \cong \bigoplus_{i=1}^r \bigoplus_{j=1}^s \ \left( \bbC \oplus \bbC^{\text{op}} \right)  \\
& \cong \bigoplus_{k=1}^{rs} \ \left( \bbC \oplus \bbC^{\text{op}} \right),
\end{align*}
and by Lemma \ref{lemma:gtnrAoplusB equals gtnrAoplusgtnrB}
and Example \ref{eg:gtnrSplusSop} we arrive
\begin{equation*}
\begin{split}
\mathfrak{g}_{2n,\rho} \left( \Cq/\Jfg \right)
\cong & \mathfrak{g}_{2n,\rho} \left( \bigoplus_{k=1}^{rs} \left( \bbC \oplus \bbC^{\text{op}} \right) \right)
\cong \bigoplus_{k=1}^{rs} \ \mathfrak{g}_{2n,\rho} \left( \bbC \oplus \bbC^{\text{op}} \right) \\
\cong & \bigoplus_{k=1}^{rs} \ \sl_{2n}(\bbC)
= \sl_{2n}(\bbC)^{\,rs}.
\end{split}
\end{equation*}

\item[(5)]
Let $(q,\thetaone,\thetatwo)=(1,-1,-1)$.
Now, as in part (4), $f(T) = \lambda_1 (T- \alpha_1)(T + \alpha_1) \cdots
(T-\alpha_r)(T + \alpha_r)$ and $g(T) = \lambda_2 (T- \beta_1)(T + \beta_1)
\cdots (T - \beta_s)(T + \beta_s)$ for some $\lambda_1, \lambda_2 \in \bbC^*$,
$r, s \geq 1$, distinct $\alpha_1, \ldots, \alpha_r \in \bbC^*$,
and distinct $\beta_1, \ldots, \beta_s \in \bbC^*$.
So,
\begin{align*}
\Cq/\Jfg
& \cong \bigoplus_{i=1}^r \bigoplus_{j=1}^s \ \frac{\Cq}{J\left( (\tone - \alpha_i)(\tone + \alpha_i),
\ (\ttwo - \beta_j) (\ttwo + \beta_j) \right)}  \\
& \cong \bigoplus_{i=1}^r \bigoplus_{j=1}^s \ \frac{\Cq}{J( \alpha_i, \ \beta_j )}
\cong \bigoplus_{i=1}^r \bigoplus_{j=1}^s \ \bbC\K  \\
& \cong \bigoplus_{k=1}^{rs} \ \bbC\K ,
\end{align*}
and by Lemma \ref{lemma:gtnrAoplusB equals gtnrAoplusgtnrB}
and Example \ref{eg:gtnrCK} we get
\begin{equation*}
\begin{split}
\mathfrak{g}_{2n,\rho} \left( \Cq/\Jfg \right)
\cong & \mathfrak{g}_{2n,\rho} \left( \bigoplus_{k=1}^{rs} \ \bbC\K \right)
\cong \bigoplus_{k=1}^{rs} \ \mathfrak{g}_{2n,\rho} \left( \bbC\K \right) \\
\cong & \bigoplus_{k=1}^{rs} \ \left( \sl_{2n}(\bbC) \oplus \sl_{2n}(\bbC) \right)
= \sl_{2n}(\bbC)^{\,2rs}.  \qedhere
\end{split}
\end{equation*}
\end{itemize}
\end{proof}

\begin{rem} \label{remarkforLemma}
If we denote the elementary quantum torus with above involution by $\C_{q,\thetaone,\thetatwo}$,
then Lemma 3.2.12 in \cite{Y1} implies that $\C_{1,-1,1} \cong \C_{1,1,-1} \cong \C_{1,-1,-1}$
and $\C_{-1,1,1} \cong \C_{-1,1,-1} \cong \C_{-1,-1,1}$.
\end{rem}

\begin{lem}\label{lemma:gtnrJf0g0 contains gtnrJfg}
Suppose $n \geq 3$. Let $(q,\theta_1,\theta_2) \in \P$ and $f(T),
g(T) \in \C[T]$ with $\deg f \deg g > 0$, $f(0)g(0) \neq 0$ and
$f,g$ satisfy
\begin{equation*}
f(\theta_1^m T)=  f(T), \; g(\theta_2^m T)=  g(T).
\end{equation*}
Let $f_0=f/(f,f')$ and $g_0=g/(g,g')$. Then
$$
\underline{\fg_{2n,\rho}}(J(f_0, g_0)) =\overline{\fg_{2n,\rho}}
(J(f_0, g_0)) \supset \overline{\fg_{2n,\rho}}(J(f, g)) =
\underline{\fg_{2n,\rho}}(J(f, g)).
$$
and $\overline{\fg_{2n,\rho}}(J(f_0, g_0))/ \overline{\fg_{2n,\rho}}
(J(f, g))$ is the radical of $\fg_{2n,\rho}(\C_q)
/\overline{\fg_{2n,\rho}} (J(f, g))$.
\end{lem}
\begin{proof}
The assumptions $f(T) = f(\thetaone^m T)$, $g(T) = g(\thetatwo^m
T)$, $\fzero = f/(f,f')$, and $\gzero = g/(g,g')$ imply that
$\fzero(\thetaone^m T) = \fzero(T)$ and $\gzero(\thetatwo^m T) =
\gzero(T)$ (see Remark \ref{remark:meaning of f(thetamT)=f(T)}).
So, by Lemma \ref{lemma:Jfg closed under involution} and Lemma
\ref{lemma:ovlgtnrJfg=undrlgtnrJfg}, we have
$\ovlgtnrJfg = \undrlgtnrJfg$ and $\ovlgtnrJfzgz =\undrlgtnrJfzgz$.
Moreover, the assumption implies $\Jfg \subset \Jfzgz$.  So we have
\begin{equation*}
\ovlgtnrJfg = \undrlgtnrJfg \subset \undrlgtnrJfzgz = \ovlgtnrJfzgz.
\end{equation*}
Hence, by Proposition \ref{thm:gtnrR/ovlgtnrJ cong gtnr(R/J)},
\begin{equation*}
\frac{ \gtnr(\Cq)/\ovlgtnrJfg }{ \ovlgtnrJfzgz/\ovlgtnrJfg }
 \cong \gtnr(\Cq)/\ovlgtnrJfzgz \cong \gtnr\left( \Cq/\Jfzgz \right).
\end{equation*}
We worked out the various possibilities for $\gtnr\left(\Cq/\Jfzgz
\right)$ in Lemma \ref{lemma:gtnr(Cq/Jfg)}, all of which
are semisimple.  So $\ovlgtnrJfzgz/\ovlgtnrJfg
\supset \calR $, where $\calR$ denotes the radical of
$\gtnr(\Cq)/\ovlgtnrJfg$.

Let us show that $\ovlgtnrJfzgz/\ovlgtnrJfg$ is, in fact, equal to
$\calR$.  We do this by showing that $\ovlgtnrJfzgz/\ovlgtnrJfg$ is
nilpotent and, hence, solvable.  Let
\begin{align*}
\ovlgtnrJfzgz^1 & = \ovlgtnrJfzgz, \\
\ovlgtnrJfzgz^k & = \left[ \ovlgtnrJfzgz^{k-1}, \ \ovlgtnrJfzgz \right], \ \text{ for } k \geq 2.
\end{align*}
It is easy to check that
\begin{equation*}
\ovlgtnrJfzgz^k \subset \overline{\gtnr}\left( \Jfzgz^k \right), \ \text{ for } k \geq 1,
\end{equation*}
where $\Jfzgz^1 = \Jfzgz$ and $\Jfzgz^k = \Jfzgz^{k-1} \cdot \Jfzgz$
for $k \geq 2$. By induction, the
elements of $\Jfzgz^k$ can be written as
$\sum\limits_{i=0}^k a_i \fzero(\tone^m)^i \gzero(\ttwo^m)^{k-i}$, for
$a_0, a_1, \ldots, a_k \in \Cq$.  By the definition of $\fzero$ and
$\gzero$, there exist positive integers $k_1$ and $k_2$ such that
$f|\fzero^{k_1}$ and $g|\gzero^{k_2}$.  Take $k > k_1 + k_2$ we have
    \begin{equation*}
        \Jfzgz^k \subset \Jfg.
    \end{equation*}
It follows that for $k$ large enough,
    \begin{equation*}
        \ovlgtnrJfzgz^k \subset \ovlgtnrJfg.
    \end{equation*}
Hence $\ovlgtnrJfzgz/\ovlgtnrJfg$ is nilpotent.
\end{proof}

\begin{thm}\label{thm:calJ contains ovlgtnrJfg}
Suppose $n \geq 3$ and $(q,\theta_1,\theta_2) \in \P$. Let $\J$ be
an ideal of $\fg_{2n,\rho}(\C_q)$ such that $\fg_{2n,\rho}(\C_q)/\J$
is finite-dimensional and semisimple. Then there exist $f, g \in
\C[T]$ with $f(0)g(0) \neq 0$, $\deg f \deg g > 0$, $(f, f') = 1$,
$(g, g') = 1$ and
\begin{equation*}
f(\theta_1^m T)=  f(T), \; g(\theta_2^m T)=  g(T),
\end{equation*}
such that
$$
\J \supset \overline{\fg_{2n,\rho}}(J(f,g))
=\underline{\fg_{2n,\rho}}(J(f, g)).
$$
\end{thm}
\begin{proof}
Let $J$ be the involutive ideal of $\Cq$ corresponding to $\calJ$ as
in Theorem \ref{thm:ideal calJ gives ideal J}.  Since
$\gtnr(\Cq)/\calJ$ is finite-dimensional and semisimple,
$J \neq \Cq$, and $\Cq/J$ is finite-dimensional.  Since $1 \notin
J$, there exists an integer $r \geq 1$ such that
\begin{equation}\label{eqn:star1}
1+J, \ \tone^m + J, \ \tone^{2m} + J, \ \ldots, \ \tone^{(r-1)m} + J
\end{equation}
is linearly independent in $\Cq/J$ but
\begin{equation*}
1+J, \ \tone^m + J, \ \tone^{2m} + J, \ \ldots, \  \tone^{rm} + J
\end{equation*}
is linearly dependent.  That is
\begin{equation}\label{eqn:star2}
\lambda_0 + \lambda_1 \tone^m + \cdots + \lambda_r \tone^{rm} \in J.
\end{equation}
for some $\lambda_0, \lambda_1, \ldots, \lambda_r \in \bbC$ with $\lambda_r \neq 0$.
We must have that $\lambda_0 \neq 0$ for if it were then (\ref{eqn:star2})
would imply that $\lambda_1 \tone^m + \cdots + \lambda_r \tone^{rm} \in J$ and hence
\begin{equation*}
\tone^{-m}\left( \lambda_1 \tone^m + \lambda_2 \tone^{2m} + \cdots +
\lambda_r \tone^{rm} \right) = \lambda_1 + \lambda_2 \tone^m + \cdots + \lambda_r \tone^{(r-1)m},
\end{equation*}
lies in $J$. Contradiction.  Similarly there is a positive integer $s$
such that $1+J, \ \ttwo^m + J, \ \ttwo^{2m} + J, \ \ldots, \
\ttwo^{(s-1)m} + J$ are linearly independent, and there exist
$\mu_0, \mu_1, \ldots, \mu_s$ with $\mu_0 \mu_s\neq 0$ such that
\begin{equation*}
\mu_0 + \mu_1 \ttwo^m + \cdots +\mu_s \ttwo^{sm} \in J.
\end{equation*}
Let
\begin{equation*}
\tildef(T) = \lambda_0 + \lambda_1 T + \cdots + \lambda_r T^r \; \text{ and }
\; \tildeg(T) = \mu_0 + \mu_1 T + \cdots +\mu_s T^s.
\end{equation*}
Since $\tildef(\tone^m)$ and $\overline{ \tildef(\tone^m) } =
\tildef\left( \overline{\tone}^{\,m} \right) = \tildef\left(
\theta_1^m \tone^m \right)$ both lie in $J$, so does their
difference
\begin{equation*}
\overline{ \tildef(\tone^m) } \, - \, \tildef(\tone^m)
\ = \ \tildef\left( \theta_1^m \tone^m \right) \, - \,
\tildef(\tone^m) \ = \ \sum_{i=1}^r \lambda_i (\thetaone^{im} - 1) \tone^{im}.
\end{equation*}
This implies that $\thetaone^{im} = 1$ whenever $\lambda_i \neq 0$.
Hence $\tildef\left( \thetaone^m T \right) = \tildef(T)$. Similarly
$\tildeg\left( \thetatwo^m T \right) = \tildeg(T)$.  By Theorem
\ref{thm:ideal calJ gives ideal J}, Lemma \ref{lemma:Jfg closed
under involution}, and Lemma \ref{lemma:ovlgtnrJfg=undrlgtnrJfg}, we
get
\begin{equation*}
\calJ \ \supset \ \undrlgtnrJ \ \supset \ \undrlgtnrJtftg \ = \ \ovlgtnrJtftg.
\end{equation*}

Let $f = \tildef/(\tildef, \tildef')$ and $g = \tildeg/(\tildeg, \tildeg')$.
Then $f(0)g(0) \neq 0$, $\deg (f) \deg (g) > 0$, $(f,f') = (g,g')= 1$,
$f(\thetaone^m T) = f(T)$, and $g(\thetatwo^m T) = g(T)$.  By Lemma
\ref{lemma:gtnrJf0g0 contains gtnrJfg}, $\ovlgtnrJfg/\ovlgtnrJtftg$
is the radical of $\gtnr(\Cq)/\ovlgtnrJtftg$.  On the other hand,
the semisimplicity of
\begin{equation*}
\gtnr(\Cq)/\calJ \ \cong \ \frac{ \gtnr(\Cq)/\ovlgtnrJtftg }{ \calJ/\ovlgtnrJtftg }
\end{equation*}
implies that
\begin{equation*}
\calJ/\ovlgtnrJtftg \ \supset \ \ovlgtnrJfg/\ovlgtnrJtftg.
\end{equation*}
That is, $\calJ \supset \ovlgtnrJfg$.
\end{proof}

\section{Finite dimensional irreducible representations of $\fg_{2n,\rho}(\C_q)$}

Let $V$ be a finite-dimensional irreducible module over
$\fg_{2n,\rho}(\C_q)$. Then ${\rm Ann}(V) = \{x \in
\fg_{2n,\rho}(\C_q) |x.V = 0\}$ is an ideal of $\fg_{2n,\rho}(\C_q)$
and $V$ can be regarded as a faithful irreducible representation of
$\fg_{2n,\rho}(\C_q)/{\rm Ann}(V)$. And there is a monomorphism of
Lie algebras $\fg_{2n,\rho}(\C_q)/{\rm Ann}(V) \hookrightarrow
\gl(V )$, i.e., $\fg_{2n,\rho}(\C_q)/{\rm Ann}(V)$ can be regarded
as a subalgebra of $\gl(V)$. Moreover, by Proposition
\ref{propn:gtnrR is perfect}, we have that
$$
\fg_{2n,\rho}(\C_q)/{\rm Ann}(V)=[\fg_{2n,\rho}(\C_q)/{\rm
Ann}(V),\fg_{2n,\rho}(\C_q)/{\rm Ann}(V)] \hookrightarrow
\sl(V), \text{ for } n \geq 3.
$$

\begin{thm}
Let $n \geq 3$. If $V$ is an irreducible representation of
$\fg_{2n,\rho}(\C_q)$ and $\dim V \leq 1$, then
$\fg_{2n,\rho}(\C_q)$ acts trivially on $V$.
\end{thm}
\begin{proof}
Since $\gtnr(\Cq)$ is perfect, the Lie algebra homomorphism
corresponding to the given representation is into $\sl(V)$.
But $\sl(V) = \{0\}$ when $\dim V \leq 1$.   Hence
$\gtnr(\Cq)$ acts trivially on $V$.
\end{proof}

\begin{lem}[Proposition 19.1 in \cite{H}] \label{lemma:Humphreys proposition}
Let $\fg \subset \sl(V)$ ($V$ finite-dimensional) be a nonzero
Lie algebra acting irreducibly on $V$. Then $\fg$ is semisimple.
\qed
\end{lem}

\begin{thm}
Let $n \geq 3$. $V$ is a finite-dimensional irreducible
representation of $\fg_{2n,\rho}(\C_q)$ and $\dim V \geq 2$.
\begin{enumerate}
\item
If $(q,\theta_1,\theta_2) \in \P$ with $q=-1$ or $(q,\theta_1,\theta_2)=(1,1,1)$,
then there exists $k \geq 0$ such that
\begin{equation*}
\fg_{2n,\rho}(\C_q)/{\rm Ann}(V) \cong \fg_{2mn,\theta_1\rho}(\C)^{k}
\end{equation*}
where $m$ is the order of $q$ as a primitive root of unit.
\item
If $(q,\theta_1,\theta_2)=(1,1,-1)$ or $(1,-1,-1)$, then there exists $k \geq 0$ such that
\begin{equation*}
\fg_{2n,\rho}(\C_q)/{\rm Ann}(V) \cong \sl_{2n}(\C)^k.
\end{equation*}
\end{enumerate}
\end{thm}
\begin{proof}
Since $2 \leq \dim V \leq +\infty$ and $V$ is irreducible,
$0 \neq \gtnr(\Cq)/\AnnV \hookrightarrow \sl(V)$.
By Lemma
\ref{lemma:Humphreys proposition}, $\gtnr(\Cq)/\AnnV$ is semisimple.
So, by Theorem \ref{thm:calJ contains ovlgtnrJfg}, there exist $f,g
\in \bbC[T]$ with
\[
f(0)g(0) \neq 0, \deg f \deg g > 0, (f,f') =(g,g') = 1,  f(\thetaone^m T)= f(T), g(\thetatwo^m T) = g(T)
\]
and
\[
\AnnV \supset \ovlgtnrJfg = \undrlgtnrJfg.
\]
Hence $\gtnr(\Cq)/\AnnV$ is isomorphic to a quotient of
$\gtnr(\Cq)/\ovlgtnrJfg$ which, by Proposition \ref{thm:gtnrR/ovlgtnrJ
cong gtnr(R/J)}, is isomorphic to $\gtnr\left( \Cq/\Jfg \right)$.
Lemma \ref{lemma:gtnr(Cq/Jfg)}, thus, tells us that
$\gtnr(\Cq)/\AnnV$ is, depending on $(q,\thetaone,\thetatwo)$,
isomorphic to $\mathfrak{g}_{2mn,\theta_1\rho}(\bbC)^k$ or
$\sl_{2n}(\bbC)^k$, for some $k \geq 0$.
\end{proof}

\begin{rem}  \label{remarkforTheorem}
By Section 3 and Section 4 in \cite{v}, we know that $\gtnr(\C_{q,\thetaone,\thetatwo})$ is isomorphic to
some double (twisted) loop algebra.
For example, we have
\begin{enumerate}
\item
$\gtnr(\C_{1,1,1})$ is isomorphic to $\fg_{2n,\rho}(\C) \otimes \C[t_1^{\pm 1},t_2^{\pm 1}]$, which is
the untwisted double loop algebra of $\fg_{2n,\rho}(\C)$ of type $X_n^{(1,1)}$ for $X_n$ equal to $C_n$ or $D_n$;

\item
$\gtnr(\C_{1,1,-1})$ is isomorphic to the twisted loop algebra of bi-affine algebra of type $A_{2n-1}^{(1,2)}$
which is a subalgebra of the double loop algebra $\sl_{2n}(\C) \otimes \C[t_1^{\pm 1},t_2^{\pm 1}]$;

\item
$\fg_{2n,-1}(\C_{-1,1,1})$ is isomorphic to the twisted loop algebra of bi-affine algebra of type $C_n^{(1,1)*}$
which is a subalgebra of the double loop algebra $\fg_{4n,-1}(\C) \otimes \C[t_1^{\pm 1},t_2^{\pm 1}]$, and so on.
\end{enumerate}
Moreover, the isomorphisms make our results coincide with the results obtained in \cite{L}.
\end{rem}

\section*{Acknowledgments}

Main part of the research was carried out during the visit of the second author at York University in 2009.
The hospitality and financial support of York University are gratefully acknowledged.
The second author is partially supported by the Recruitment Program of Global Youth Experts of China,
by the start-up funding from University of Science and Technology of China, by the Fundamental Research Funds for the Central Universities,
by NSF of China (Grant 11401551, 11471294, 11771410). The third author is partially supported by NSERC of Canada.

\end{document}